\documentclass[11pt]{article}

\usepackage[margin=1in]{geometry}
\usepackage[T1]{fontenc}
\usepackage[utf8]{inputenc}
\usepackage{lmodern}
\usepackage{microtype}
\usepackage{amsmath,amssymb,amsthm,mathtools}
\usepackage{bm}
\usepackage{enumitem}
\usepackage{booktabs}
\usepackage{graphicx}
\usepackage{hyperref}
\usepackage[nameinlink,capitalize,noabbrev]{cleveref}

\hypersetup{hidelinks}
\numberwithin{equation}{section}
\allowdisplaybreaks

\newtheorem{theorem}{Theorem}[section]
\newtheorem{lemma}[theorem]{Lemma}
\newtheorem{proposition}[theorem]{Proposition}
\newtheorem{corollary}[theorem]{Corollary}

\theoremstyle{remark}
\newtheorem{remark}[theorem]{Remark}

\crefname{theorem}{Theorem}{Theorems}
\Crefname{theorem}{Theorem}{Theorems}
\crefname{lemma}{Lemma}{Lemmas}
\Crefname{lemma}{Lemma}{Lemmas}
\crefname{proposition}{Proposition}{Propositions}
\Crefname{proposition}{Proposition}{Propositions}
\crefname{corollary}{Corollary}{Corollaries}
\Crefname{corollary}{Corollary}{Corollaries}
\crefname{remark}{Remark}{Remarks}
\Crefname{remark}{Remark}{Remarks}
\crefname{assumption}{Assumption}{Assumptions}
\Crefname{assumption}{Assumption}{Assumptions}

\newcommand{\R}{\mathbb{R}}

\newcommand{\ii}{\mathrm{i}}
\newcommand{\T}{\mathcal{T}}
\newcommand{\E}{\mathcal{E}}
\newcommand{\Pp}{\mathbb{P}}
\newcommand{\CR}{\mathrm{CR}}
\newcommand{\Gh}{G_h}
\newcommand{\Lh}{\mathcal L_h}
\newcommand{\bn}{\bm n}
\newcommand{\bq}{\bm q}
\newcommand{\br}{\bm r}
\newcommand{\btau}{\bm \tau}

\newcommand{\bv}{\bm v}

\newcommand{\epsq}{\varepsilon_h^q}
\newcommand{\epsu}{\varepsilon_h^u}
\newcommand{\epsuh}{\varepsilon_h^{\widehat u}}

\title{Mesh-Uniform Stability and Error Estimates for HDG Discretizations of the Helmholtz Equation}
\author{Yukun Yue\thanks{Department of Mathematics, University of Wisconsin--Madison, Madison, WI 53706, USA. Email: \texttt{yyue24@wisc.edu}.}}
\date{}

\begin{document}
\maketitle

\begin{abstract}
We revisit the hybridizable discontinuous Galerkin method of Cui and Zhang for the Helmholtz equation with first-order absorbing boundary condition. Their analysis proves stability without imposing a mesh constraint coupling $h$ and $\kappa$, but the explicit stability bound still contains negative powers of the mesh size. We prove that this mesh-dependent blow-up is not intrinsic to the HDG discretization. Under the same geometric and mesh framework as in the reference analysis, for fixed polynomial degree and for stabilization parameters uniformly bounded above and below, we establish a generalized stability estimate whose constant is independent of the mesh size. For every prescribed compact interval $0<\kappa_0\le\kappa\le\kappa_1<\infty$, the stability constant can be chosen locally uniformly with respect to $\kappa$. The proof replaces the Rellich-identity and inverse-estimate argument by a compactness argument based on the HDG discrete distributional gradient and the Crouzeix--Raviart lifting mechanism. As consequences, we obtain well-posedness, convergence for minimal $L^2$ data, and projection-based error estimates in which the negative powers of the mesh size appearing in the previous theory are removed. Numerical experiments confirm the predicted mesh-uniform behavior and show a clear contrast with the mesh-dependent growth suggested by the previous explicit bound.
\end{abstract}

\section{Introduction}

The Helmholtz equation is a prototypical noncoercive time-harmonic wave problem: the associated variational form contains the negative mass term $-\kappa^2(u,v)$ and is not coercive on $H^1(\Omega)$ in the standard elliptic variational sense; see, for example, \cite[Chapter~6]{Evans2010}. Its finite element approximation is delicate because the stability of the continuous resolvent, the pollution effect, and the preasymptotic behavior of the discrete scheme all depend on the wave number; see, for example, \cite{BabuskaSauter1997,FengWu2009,IhlenburgBabuska1995}. The regularity of the solution, including the high-frequency behavior of corner singularities on nonsmooth domains, further affects the attainable convergence rates \cite{chaumont2018high}. The Helmholtz equation has also motivated a broad computational and inverse-problem literature, including absorbing-boundary and iterative solvers \cite{bayliss1983iterative}, fast multipole acceleration \cite{gumerov2005fast}, and high-frequency inverse scattering asymptotics \cite{chen2023high}.

Hybridizable discontinuous Galerkin (HDG) methods are attractive in this setting because they retain a local discontinuous structure while reducing the global unknowns to skeletal traces. General HDG frameworks, projection tools, and a posteriori error estimation are developed in \cite{ainsworth2018fully,Cockburn2009,CockburnGopalakrishnanSayas2010,Du2019}. Beyond wave problems, HDG methods have been analyzed for convection-dominated diffusion \cite{fu2015convection}, for Stokes and continuum mechanics problems including linear elasticity \cite{fu2015elasticity,Nguyen2012,Nguyen2010}, for incompressible Navier--Stokes and pressure-robust formulations \cite{Cesmelioglu2016,Kirk2019}, for Cahn--Hilliard equations \cite{Chen2023}, and for miscible displacement under minimal regularity \cite{Kirk2026}. For Helmholtz problems, a range of hybrid and hybridizable formulations has been studied, including a multiscale hybrid-mixed method for heterogeneous media \cite{chaumont2020multiscale} and an HDG method with characteristic variables \cite{modave2023hybridizable}. In particular, \cite{GriesmaierMonk2011} analyzed an HDG method for the interior Dirichlet problem, and Cui and Zhang \cite{CuiZhang2014} analyzed an HDG method for the Helmholtz equation with first-order absorbing boundary condition in two and three dimensions. The condition $h\kappa^2$ sufficiently small is mentioned here because it is the mesh restriction under which the earlier Dirichlet HDG analysis of \cite{GriesmaierMonk2011} obtains its stability and optimal-convergence result; it is stronger than the basic resolution condition $h\kappa\lesssim1$. Cui and Zhang's absorbing-boundary HDG result is therefore important: the method is stable and well posed for every $h>0$ and $\kappa>0$, without imposing such a mesh restriction, and projection-based $L^2$ error estimates are obtained for arbitrary polynomial degree.

The explicit stability factor in \cite{CuiZhang2014}, however, contains negative powers of the mesh size. More precisely, \cite[Theorem~3.1, equation~(3.3)]{CuiZhang2014} gives the generalized stability factor
\begin{equation}\label{eq:intro-CZ-constant}
C_{\rm CZ}(h,\kappa,\tau)
=\{\kappa^2\tau_{\min}^{-1}+\kappa+\tau_{\min}^{-1}h^{-3}+\tau_{\max}h^{-1}\}^2+1.
\end{equation}
Here $h$ is the maximal element diameter, $\tau$ is the facewise positive HDG stabilization parameter, and $\tau_{\min}$ and $\tau_{\max}$ denote its minimum and maximum over all mesh faces. Thus the upper bound deteriorates as $h\to0$, even though the method is stable for every mesh size. This is not merely a cosmetic issue, because the same stability factor propagates into the projection error estimate and hence into the final $L^2$ error bounds. For the common choice $\tau=\kappa$ \cite{GriesmaierMonk2011}, the mesh-singular part in the braces in \eqref{eq:intro-CZ-constant} contains $\kappa^{-1}h^{-3}+\kappa h^{-1}$; for each fixed wave number this is still dominated by an $h^{-3}$ contribution as $h\to0$, and the squared stability factor contains the corresponding higher-order mesh singularity.

The source of the mesh-dependent powers is the proof mechanism. The argument in \cite{CuiZhang2014} uses elementwise Rellich identities with multiplier-type discrete test functions, followed by inverse estimates for residuals and jumps. In particular, the tangential jump estimates generate negative powers of $h$ which are then inherited by the final constant. The goal of this paper is to show that these negative powers are a proof artifact rather than an intrinsic stability feature of the HDG Helmholtz operator.

Our proof is based on an HDG compactness mechanism. The central object is an HDG discrete distributional gradient, denoted by $\Gh(u_h,\widehat u_h)$ and defined in Section~\ref{sec:setting}. Here $u_h$ is the element unknown and $\widehat u_h$ is the single-valued skeletal trace unknown of the HDG method. The operator $\Gh$ combines the broken gradient of $u_h$ with the mismatch between $u_h$ and $\widehat u_h$. The compactness argument uses a Crouzeix--Raviart lifting of face averages of the skeletal variable; this mechanism is closely related to the discrete Poincar\'e and trace inequalities developed in \cite{Yue2026BIT} and to the minimal-regularity HDG compactness framework of \cite{JiangWalkingtonYue2025}.

The main contributions are the following.
\begin{enumerate}[label=(\roman*),leftmargin=2.25em]
\item We prove an $h$-uniform generalized stability estimate for the auxiliary HDG system with an additional right-hand side $Q$ in the first, flux-defining equation. The proof first establishes a small-mesh estimate by compactness and then combines it with the coarse-mesh consequence of \cite[Theorem~3.1]{CuiZhang2014}. The estimate controls $\bq_h$, $u_h$, the boundary trace $\widehat u_h$, and the stabilization jump $u_h-\widehat u_h$.
\item The estimate is locally uniform in the wave number: for every compact interval $0<\kappa_0\le\kappa\le\kappa_1<\infty$, the stability constant can be chosen uniformly for $\kappa$ in that interval, provided the stabilization is uniformly bounded above and below on the same interval.
\item Applying this generalized stability estimate to the HDG projection error equations yields projection-based error estimates in which the negative powers of $h$ in the previous stability factor are removed.
\item We include numerical experiments that verify the expected convergence rates, test the generalized $Q$-forcing stability estimate, and compare the measured stability ratio with the mesh-singular reference curve from \cite{CuiZhang2014}.
\end{enumerate}
Throughout the paper the polynomial degree $p$ is fixed. Constants may depend on $p$, the domain, the shape-regularity class, the prescribed wave-number interval, and uniform upper and lower bounds for the stabilization parameter, but not on $h$.

\section{Problem setting and HDG formulation}\label{sec:setting}

\subsection{Continuous problem and assumptions}
Let $\Omega\subset\R^d$, $d=2,3$, be a bounded strictly star-shaped polygonal or polyhedral domain with respect to a point $x_\Omega\in\Omega$. Thus there exists $c_\Omega>0$ such that
\begin{equation}\label{eq:star-shaped}
(x-x_\Omega)\cdot \bn_\Omega \ge c_\Omega
\qquad\text{for }x\in\partial\Omega,
\end{equation}
where $\bn_\Omega$ denotes the exterior unit normal. This is the same geometric assumption used in \cite{CuiZhang2014}.

For a wave number $\kappa>0$, we consider
\begin{subequations}\label{eq:helmholtz-primal}
\begin{align}
-\Delta u-\kappa^2 u&=f &&\text{in }\Omega,\label{eq:helmholtz-primal-a}\\
\partial_{\bn}u+\ii\kappa u&=g &&\text{on }\partial\Omega.\label{eq:helmholtz-primal-b}
\end{align}
\end{subequations}
Equivalently, setting $\bq=-\nabla u$ gives
\begin{subequations}\label{eq:helmholtz-mixed}
\begin{align}
\bq+\nabla u&=0 &&\text{in }\Omega,\\
\nabla\cdot\bq-\kappa^2u&=f &&\text{in }\Omega,\\
-\bq\cdot\bn+\ii\kappa u&=g &&\text{on }\partial\Omega.
\end{align}
\end{subequations}

For measurable sets $D\subset\Omega$ and $\Sigma\subset\partial\Omega$, we use complex $L^2$ pairings linear in the first argument and conjugate-linear in the second:
\[
(u,v)_D=\int_D u\overline v\,dx,
\qquad
\langle u,v\rangle_\Sigma=\int_\Sigma u\overline v\,ds.
\]
All real-valued finite element estimates used below are applied componentwise to real and imaginary parts.

\subsection{Meshes, spaces, and stabilization}
Let $\{\T_h\}_{h>0}$ be a shape-regular family of conforming simplicial triangulations of $\Omega$. Thus any two distinct elements intersect only in a common subsimplex, and the boundary mesh conforms to $\partial\Omega$. For each element $K\in\T_h$, let $h_K=\operatorname{diam}(K)$ and let $\rho_K$ be its inradius. We use a uniform shape-regularity assumption for the whole mesh family, namely
\[
\sup_h\sup_{K\in\T_h}\frac{h_K}{\rho_K}<\infty,
\qquad
h=\max_{K\in\T_h}h_K.
\]
The set of all faces is denoted by $\E_h$, with interior and boundary subsets $\E_h^i$ and $\E_h^b$. A generic face is denoted by $F$. For elementwise and skeleton norms we write
\[
\|v\|_{\T_h}^2=\sum_{K\in\T_h}\|v\|_{L^2(K)}^2,
\qquad
\|\mu\|_{\partial\T_h}^2=\sum_{K\in\T_h}\|\mu\|_{L^2(\partial K)}^2.
\]
On an interior face the skeleton norm counts the two traces from the two adjacent elements. On a boundary face it counts the single trace from the unique adjacent element.

Fix $p\ge0$. The HDG spaces are
\[
V_h=\{\bv\in L^2(\Omega)^d:\ \bv|_K\in \Pp_p(K)^d\ \forall K\in\T_h\},
\]
\[
W_h=\{w\in L^2(\Omega):\ w|_K\in \Pp_p(K)\ \forall K\in\T_h\},
\qquad
M_h=\{\mu\in L^2(\E_h):\ \mu|_F\in \Pp_p(F)\ \forall F\in\E_h\}.
\]
Let $P_W$, $P_M$, and $P_V$ denote the $L^2$ projections onto $W_h$, $M_h$, and $V_h$, respectively.

The stabilization $\tau$ is positive and constant on each face. For fixed-wave-number statements, we assume
\begin{equation}\label{eq:tau-bounds-fixed}
0<\tau_{\min}\le \tau|_F\le \tau_{\max}<\infty,
\end{equation}
with bounds independent of $h$. For locally uniform wave-number statements, we assume that for every compact interval $[\kappa_0,\kappa_1]\subset(0,\infty)$ there exist $0<\tau_0\le \tau_1<\infty$ such that
\begin{equation}\label{eq:tau-bounds-uniform}
\tau_0\le \tau|_F\le \tau_1
\qquad\text{for all faces }F\text{ and all }\kappa\in[\kappa_0,\kappa_1].
\end{equation}
For example, the choice $\tau=c_\tau\kappa$ with fixed $c_\tau>0$ satisfies \eqref{eq:tau-bounds-uniform} on every compact wave-number interval.

\subsection{The auxiliary generalized HDG system}
We first state a slightly enlarged HDG system, following \cite{CuiZhang2014}. The symbol $Q$ is not an additional physical forcing and does not modify the actual HDG method used to solve the Helmholtz equation. It is an auxiliary vector-valued right-hand side in the first, flux-defining equation \eqref{eq:hdg-general-a} below. In the actual Helmholtz discretization one sets $Q=0$. We keep $Q$ because the same algebraic system is satisfied by the projected error in Section~\ref{sec:error}; there, $Q=\pi\bq-\bq$ is the residual caused by replacing the exact flux $\bq$ by its local HDG projection $\pi\bq$ in the first mixed equation.

The auxiliary generalized system seeks $(\bq_h,u_h,\widehat u_h)\in V_h\times W_h\times M_h$ such that
\begin{subequations}\label{eq:hdg-general}
\begin{align}
(\bq_h,\btau_h)_{\T_h}-(u_h,\nabla\cdot\btau_h)_{\T_h}
+\langle \widehat u_h,\btau_h\cdot\bn\rangle_{\partial\T_h}
&=(Q,\btau_h)_{\T_h},\label{eq:hdg-general-a}\\
-(\bq_h,\nabla v_h)_{\T_h}-\kappa^2(u_h,v_h)_{\T_h}
+\langle\widehat{\bq}_h\cdot\bn,v_h\rangle_{\partial\T_h}
&=(f,v_h)_{\T_h},\label{eq:hdg-general-b}\\
\langle-\widehat{\bq}_h\cdot\bn+\ii\kappa\widehat u_h,\mu_h\rangle_{\partial\Omega}
&=\langle g,\mu_h\rangle_{\partial\Omega},\label{eq:hdg-general-c}\\
\langle\widehat{\bq}_h\cdot\bn,\mu_h\rangle_{\partial\T_h\setminus\partial\Omega}
&=0,\label{eq:hdg-general-d}
\end{align}
\end{subequations}
for all $(\btau_h,v_h,\mu_h)\in V_h\times W_h\times M_h$, where
\begin{equation}\label{eq:numerical-flux}
\widehat{\bq}_h=\bq_h+\ii\tau(u_h-\widehat u_h)\bn
\qquad\text{on }\partial\T_h.
\end{equation}

\subsection{The HDG discrete distributional gradient and pair formulation}
For $(v_h,\widehat v_h)\in W_h\times M_h$, define the HDG discrete distributional gradient $\Gh(v_h,\widehat v_h)\in V_h$ by
\begin{equation}\label{eq:Gh-definition}
(\Gh(v_h,\widehat v_h),\br_h)_{\T_h}
=(\nabla_h v_h,\br_h)_{\T_h}
+\langle\widehat v_h-v_h,\br_h\cdot\bn\rangle_{\partial\T_h}
\qquad\forall \br_h\in V_h.
\end{equation}
This definition is well posed because $(\cdot,\cdot)_{\T_h}$ is an inner product on the finite-dimensional space $V_h$; hence the Riesz representation theorem gives a unique element $\Gh(v_h,\widehat v_h)\in V_h$.

Elementwise integration by parts gives
\[
(\nabla_h v_h,\br_h)_{\T_h}
=-(v_h,\nabla\cdot\br_h)_{\T_h}
+\langle v_h,\br_h\cdot\bn\rangle_{\partial\T_h},
\]
and therefore
\begin{equation}\label{eq:Gh-equivalent}
(\Gh(v_h,\widehat v_h),\br_h)_{\T_h}
=-(v_h,\nabla\cdot\br_h)_{\T_h}
+\langle\widehat v_h,\br_h\cdot\bn\rangle_{\partial\T_h}
\qquad\forall \br_h\in V_h.
\end{equation}
Thus \eqref{eq:hdg-general-a} is equivalent to
\begin{equation}\label{eq:q-plus-Gh}
\bq_h+\Gh(u_h,\widehat u_h)=P_VQ.
\end{equation}

We will also use the consistency of this definition with the usual distributional gradient. If $v\in H^1(\Omega)$, then, for every $\br_h\in V_h$,
\begin{align*}
(\Gh(P_Wv,P_Mv),\br_h)_{\T_h}
&=-(P_Wv,\nabla\cdot\br_h)_{\T_h}
+\langle P_Mv,\br_h\cdot\bn\rangle_{\partial\T_h}\\
&=-(v,\nabla\cdot\br_h)_{\T_h}
+\langle v,\br_h\cdot\bn\rangle_{\partial\T_h}\\
&=(\nabla v,\br_h)_\Omega.
\end{align*}
Consequently,
\begin{equation}\label{eq:Gh-projection-gradient}
\Gh(P_Wv,P_Mv)=P_V\nabla v.
\end{equation}

The remaining equations can be written in a form using test pairs.
\begin{lemma}\label{lem:pair-form}
A triple $(\bq_h,u_h,\widehat u_h)$ satisfying \eqref{eq:hdg-general-b}--\eqref{eq:hdg-general-d} also satisfies, for all $(v_h,\widehat v_h)\in W_h\times M_h$,
\begin{align}
&-(\bq_h,\Gh(v_h,\widehat v_h))_{\T_h}
-\kappa^2(u_h,v_h)_{\T_h}
+\ii\langle\tau(u_h-\widehat u_h),v_h-\widehat v_h\rangle_{\partial\T_h}
+\ii\kappa\langle\widehat u_h,\widehat v_h\rangle_{\partial\Omega}
\notag\\
&\hspace{7cm}= (f,v_h)_{\T_h}+\langle g,\widehat v_h\rangle_{\partial\Omega}.
\label{eq:pair-form}
\end{align}
Conversely, \eqref{eq:pair-form} implies \eqref{eq:hdg-general-b}--\eqref{eq:hdg-general-d}.
\end{lemma}

\begin{proof}
Starting from \eqref{eq:hdg-general-b}, write
\[
\langle\widehat{\bq}_h\cdot\bn,v_h\rangle_{\partial\T_h}
=
\langle\widehat{\bq}_h\cdot\bn,v_h-\widehat v_h\rangle_{\partial\T_h}
+
\langle\widehat{\bq}_h\cdot\bn,\widehat v_h\rangle_{\partial\T_h}.
\]
The interior part of the second term vanishes by \eqref{eq:hdg-general-d}, while \eqref{eq:hdg-general-c} gives
\[
\langle\widehat{\bq}_h\cdot\bn,\widehat v_h\rangle_{\partial\Omega}
=
\ii\kappa\langle\widehat u_h,\widehat v_h\rangle_{\partial\Omega}
-
\langle g,\widehat v_h\rangle_{\partial\Omega}.
\]
Using \eqref{eq:numerical-flux},
\[
\langle\widehat{\bq}_h\cdot\bn,v_h-\widehat v_h\rangle_{\partial\T_h}
=
\langle\bq_h\cdot\bn,v_h-\widehat v_h\rangle_{\partial\T_h}
+
\ii\langle\tau(u_h-\widehat u_h),v_h-\widehat v_h\rangle_{\partial\T_h}.
\]
Finally, by \eqref{eq:Gh-definition},
\[
-(\bq_h,\Gh(v_h,\widehat v_h))_{\T_h}
=-(\bq_h,\nabla_hv_h)_{\T_h}
+
\langle\bq_h\cdot\bn,v_h-\widehat v_h\rangle_{\partial\T_h}.
\]
Combining these identities proves \eqref{eq:pair-form}. The converse follows by choosing $\widehat v_h=0$, then $v_h=0$ on interior and boundary faces separately.
\end{proof}

\section{HDG compactness tools}\label{sec:compactness}

This section records the compactness input used in the stability proof and fixes the notation for the Crouzeix--Raviart lifting. The point is that the compactness is not obtained from the broken gradient alone; it is obtained from the HDG discrete distributional gradient $\Gh$ together with the stabilization jump.

\subsection{Embedding of continuous functions}
For $v\in H^1(\Omega)$, define the HDG embedding
\begin{equation}\label{eq:iota-definition}
\iota_hv=(P_Wv,P_Mv)\in W_h\times M_h.
\end{equation}

\begin{lemma}\label{lem:iota-consistency}
For every $v\in H^1(\Omega)$,
\begin{align}
P_Wv&\to v &&\text{strongly in }L^2(\Omega),\label{eq:PW-conv}\\
P_Mv&\to v|_{\partial\Omega} &&\text{strongly in }L^2(\partial\Omega),\label{eq:PM-conv}\\
\Gh(P_Wv,P_Mv)&\to \nabla v &&\text{strongly in }L^2(\Omega)^d.\label{eq:Gh-iota-conv}
\end{align}
Moreover,
\begin{equation}\label{eq:iota-jump-small}
\|P_Wv-P_Mv\|_{\partial\T_h}\le C h^{1/2}\|\nabla v\|_{L^2(\Omega)}.
\end{equation}
The constant is independent of $h$ for shape-regular mesh families.
\end{lemma}

\begin{proof}
The first two convergence statements are the standard strong convergence of $L^2$ projections on shape-regular meshes. The identity \eqref{eq:Gh-projection-gradient} gives
\[
\Gh(P_Wv,P_Mv)=P_V\nabla v,
\]
and the strong convergence of $P_V\nabla v$ to $\nabla v$ gives \eqref{eq:Gh-iota-conv}. Finally, the trace estimate \eqref{eq:iota-jump-small} follows by applying the elementwise trace inequality and the approximation properties of the element and face $L^2$ projections; this is the HDG embedding estimate of \cite[Lemma~4.3]{JiangWalkingtonYue2025}.
\end{proof}

\subsection{Crouzeix--Raviart lifting and compactness}
For a facewise polynomial $\widehat w_h\in M_h$, define its face average by
\[
\overline{\widehat w_h}|_F=\frac1{|F|}\int_F\widehat w_h\,ds
\qquad \forall F\in\E_h.
\]
Let $\Lh\widehat w_h\in\CR(\T_h)$ be the Crouzeix--Raviart function whose degree of freedom on each face $F$ is this average; equivalently,
\[
\frac1{|F|}\int_F \Lh\widehat w_h\,ds
=
\overline{\widehat w_h}|_F.
\]
Since a linear function has its face average equal to its value at the face centroid, this is the usual CR lifting of the facewise constant function $\overline{\widehat w_h}$. Thus the lifting is never applied directly to the full polynomial $\widehat w_h$, but only to its face average. For further details on the CR lifting, we refer the reader to Appendix \ref{app:compactness}.

\begin{theorem}\label{thm:HDG-compactness}
Let $\{\T_h\}_{h>0}$ be a shape-regular family of conforming simplicial meshes of a bounded Lipschitz polyhedral domain, with $h\to0$. Let $(w_h,\widehat w_h)\in W_h\times M_h$ be a sequence such that
\begin{equation}\label{eq:compactness-bound}
\|w_h\|_{L^2(\Omega)}+
\|\Gh(w_h,\widehat w_h)\|_{L^2(\Omega)}+
\|w_h-\widehat w_h\|_{L^2(\partial\T_h)}
\le C.
\end{equation}
Then there are a subsequence, not relabeled, and a function $w\in H^1(\Omega)$ such that
\begin{align}
w_h&\to w &&\text{strongly in }L^2(\Omega),\label{eq:compact-u}\\
\Gh(w_h,\widehat w_h)&\rightharpoonup \nabla w &&\text{weakly in }L^2(\Omega)^d,\label{eq:compact-G}\\
\widehat w_h|_{\partial\Omega}&\rightharpoonup w|_{\partial\Omega} &&\text{weakly in }L^2(\partial\Omega).\label{eq:compact-trace}
\end{align}
\end{theorem}

\begin{proof}
The proof is the compactness mechanism of \cite[Lemma~4.4]{JiangWalkingtonYue2025}, recalled here in the form needed later. The CR lifting satisfies estimates of the form
\[
\|\nabla_h\Lh\widehat w_h\|_{L^2(\Omega)}
\le
C\|\Gh(w_h,\widehat w_h)\|_{L^2(\Omega)}
\]
and
\[
\|w_h-\Lh\widehat w_h\|_{L^2(\Omega)}
\le
C\left(h\|\Gh(w_h,\widehat w_h)\|_{L^2(\Omega)}
+h^{1/2}\|w_h-\widehat w_h\|_{L^2(\partial\T_h)}\right).
\]
Thus $\Lh\widehat w_h$ is bounded in the broken $H^1$ norm of the CR space, while $w_h-\Lh\widehat w_h$ vanishes in $L^2(\Omega)$ along sequences with $h\to0$. The compactness of the CR space yields a subsequence and $w\in H^1(\Omega)$ such that $\Lh\widehat w_h\to w$ strongly in $L^2(\Omega)$ and weakly through its broken gradient. The preceding difference estimate then gives $w_h\to w$ strongly in $L^2(\Omega)$, and the definition of $\Gh$ identifies the weak limit of $\Gh(w_h,\widehat w_h)$ with $\nabla w$. The boundary trace statement follows from the CR trace compactness and the control of $w_h-\widehat w_h$ on boundary faces.
\end{proof}

\section{Mesh-uniform generalized stability}\label{sec:stability}

This section proves the main stability estimate. We first derive an energy identity for the generalized HDG system, then prove a small-mesh stability bound by contradiction and compactness, and finally combine that bound with the coarse-mesh consequence of the estimate in \cite{CuiZhang2014} to obtain an all-mesh constant independent of $h$.

\subsection{Energy identities}
We first present a fundamental energy identity for the auxiliary generalized system.

\begin{lemma}\label{lem:general-energy}
Let $(\bq_h,u_h,\widehat u_h)$ solve \eqref{eq:hdg-general}. Then
\begin{align}
\|\bq_h\|_{\T_h}^2-\kappa^2\|u_h\|_{\T_h}^2
&=
\Re\Big((f,u_h)_{\T_h}+(\bq_h,Q)_{\T_h}+\langle g,\widehat u_h\rangle_{\partial\Omega}\Big),\label{eq:energy-real-general}\\
\|\tau^{1/2}(u_h-\widehat u_h)\|_{\partial\T_h}^2
+\kappa\|\widehat u_h\|_{\partial\Omega}^2
&=
\Im\Big((f,u_h)_{\T_h}+(\bq_h,Q)_{\T_h}+\langle g,\widehat u_h\rangle_{\partial\Omega}\Big).
\label{eq:energy-imag-general}
\end{align}
\end{lemma}

\begin{proof}
Taking $\btau_h=\bq_h$ in \eqref{eq:hdg-general-a} gives
\[
\|\bq_h\|_{\T_h}^2
-(u_h,\nabla\cdot\bq_h)_{\T_h}
+\langle\widehat u_h,\bq_h\cdot\bn\rangle_{\partial\T_h}
=(Q,\bq_h)_{\T_h}.
\]
Taking the complex conjugate, and using the convention that the $L^2$ pairing is linear in the first argument, gives
\[
\|\bq_h\|_{\T_h}^2
-(\nabla\cdot\bq_h,u_h)_{\T_h}
+\langle\bq_h\cdot\bn,\widehat u_h\rangle_{\partial\T_h}
=(\bq_h,Q)_{\T_h}.
\]
Taking $v_h=u_h$ in \eqref{eq:hdg-general-b} gives
\[
-(\bq_h,\nabla u_h)_{\T_h}
-\kappa^2\|u_h\|_{\T_h}^2
+\langle\widehat{\bq}_h\cdot\bn,u_h\rangle_{\partial\T_h}
=(f,u_h)_{\T_h}.
\]
Adding the last two identities and using elementwise integration by parts yields
\[
\|\bq_h\|_{\T_h}^2-\kappa^2\|u_h\|_{\T_h}^2
+\langle(\widehat{\bq}_h-\bq_h)\cdot\bn,u_h-\widehat u_h\rangle_{\partial\T_h}
+\langle\widehat{\bq}_h\cdot\bn,\widehat u_h\rangle_{\partial\T_h}
=(f,u_h)_{\T_h}+(\bq_h,Q)_{\T_h}.
\]
Using \eqref{eq:numerical-flux},
\[
\langle(\widehat{\bq}_h-\bq_h)\cdot\bn,u_h-\widehat u_h\rangle_{\partial\T_h}
=
\ii\|\tau^{1/2}(u_h-\widehat u_h)\|_{\partial\T_h}^2.
\]
Testing \eqref{eq:hdg-general-c}--\eqref{eq:hdg-general-d} with $\widehat u_h$ gives
\[
\langle\widehat{\bq}_h\cdot\bn,\widehat u_h\rangle_{\partial\T_h}
=
\ii\kappa\|\widehat u_h\|_{\partial\Omega}^2
-\langle g,\widehat u_h\rangle_{\partial\Omega}.
\]
Substitution gives
\[
\|\bq_h\|_{\T_h}^2-\kappa^2\|u_h\|_{\T_h}^2
+\ii\|\tau^{1/2}(u_h-\widehat u_h)\|_{\partial\T_h}^2
+\ii\kappa\|\widehat u_h\|_{\partial\Omega}^2
=(f,u_h)_{\T_h}+(\bq_h,Q)_{\T_h}+\langle g,\widehat u_h\rangle_{\partial\Omega}.
\]
Taking real and imaginary parts proves the claim.
\end{proof}

\subsection{Small-mesh stability}
The following theorem is the compactness step that removes the mesh-singular constants from the small-mesh regime. It rules out loss of stability as $h\to0$.

\begin{theorem}\label{thm:small-stability}
Let $0<\kappa_0\le \kappa\le\kappa_1<\infty$ and assume \eqref{eq:tau-bounds-uniform}. There exist $h_0>0$ and $C_{\rm sm}>0$, depending on $\Omega$, $p$, the shape-regularity constant, $\kappa_0$, $\kappa_1$, $\tau_0$, and $\tau_1$, but not on $h$ or on the particular $\kappa\in[\kappa_0,\kappa_1]$, such that for every $0<h\le h_0$ and every solution of \eqref{eq:hdg-general},
\begin{align}
&\|\bq_h\|_{\T_h}^2+
\kappa^2\|u_h\|_{\T_h}^2+
\kappa\|\widehat u_h\|_{\partial\Omega}^2+
\|\tau^{1/2}(u_h-\widehat u_h)\|_{\partial\T_h}^2
\notag\\
&\hspace{6cm}\le
C_{\rm sm}
\left(
\|f\|_{\T_h}^2+
\|g\|_{\partial\Omega}^2+
\|Q\|_{\T_h}^2
\right).
\label{eq:small-stability}
\end{align}
\end{theorem}

\begin{proof}
Suppose the result is false. Then there are $h_n\to0$, wave numbers $\kappa_n\in[\kappa_0,\kappa_1]$, stabilizations satisfying \eqref{eq:tau-bounds-uniform}, data $(f_n,g_n,Q_n)$, and solutions $(\bq_n,u_n,\widehat u_n)$ such that
\begin{equation}\label{eq:contra-normalization}
\|\bq_n\|_{\T_{h_n}}^2+
\kappa_n^2\|u_n\|_{\T_{h_n}}^2+
\kappa_n\|\widehat u_n\|_{\partial\Omega}^2+
\|\tau_n^{1/2}(u_n-\widehat u_n)\|_{\partial\T_{h_n}}^2=1,
\end{equation}
and
\begin{equation}\label{eq:data-to-zero}
\|f_n\|_{\T_{h_n}}+
\|g_n\|_{\partial\Omega}+
\|Q_n\|_{\T_{h_n}}\to0.
\end{equation}
Passing to a subsequence, $\kappa_n\to\kappa_*\in[\kappa_0,\kappa_1]$.

By Cauchy's inequality, \eqref{eq:contra-normalization}, $\kappa_n\ge\kappa_0$, and \eqref{eq:data-to-zero}, the right-hand side of \eqref{eq:energy-imag-general} tends to zero. Hence
\begin{equation}\label{eq:jump-trace-zero}
\|\tau_n^{1/2}(u_n-\widehat u_n)\|_{\partial\T_{h_n}}^2+
\kappa_n\|\widehat u_n\|_{\partial\Omega}^2\to0.
\end{equation}
Since $\tau_n\ge\tau_0$ and $\kappa_n\ge\kappa_0$, it follows that
\begin{equation}\label{eq:unweighted-jump-trace-zero}
\|u_n-\widehat u_n\|_{\partial\T_{h_n}}\to0,
\qquad
\|\widehat u_n\|_{\partial\Omega}\to0.
\end{equation}
The real part identity \eqref{eq:energy-real-general}, again using \eqref{eq:contra-normalization} and \eqref{eq:data-to-zero}, gives
\begin{equation}\label{eq:real-diff-zero}
\|\bq_n\|_{\T_{h_n}}^2-
\kappa_n^2\|u_n\|_{\T_{h_n}}^2\to0.
\end{equation}
Combining \eqref{eq:contra-normalization}, \eqref{eq:jump-trace-zero}, and \eqref{eq:real-diff-zero}, we obtain
\begin{equation}\label{eq:u-nonzero-limit}
\kappa_n^2\|u_n\|_{\T_{h_n}}^2\to\frac12,
\qquad
\|\bq_n\|_{\T_{h_n}}^2\to\frac12.
\end{equation}

By \eqref{eq:q-plus-Gh},
\[
\Gh(u_n,\widehat u_n)=P_{V_{h_n}}Q_n-\bq_n.
\]
Since $Q_n\to0$ strongly in $L^2(\Omega)^d$ and $\bq_n$ is bounded in $L^2(\Omega)^d$, the sequence $\Gh(u_n,\widehat u_n)$ is bounded in $L^2(\Omega)^d$. Together with \eqref{eq:contra-normalization} and \eqref{eq:unweighted-jump-trace-zero}, \cref{thm:HDG-compactness} gives a subsequence, not relabeled, and $u\in H^1(\Omega)$ such that
\[
u_n\to u\quad\text{strongly in }L^2(\Omega),
\qquad
\Gh(u_n,\widehat u_n)\rightharpoonup \nabla u\quad\text{weakly in }L^2(\Omega)^d.
\]
Moreover, since $\widehat u_n|_{\partial\Omega}\rightharpoonup u|_{\partial\Omega}$ by compactness and $\widehat u_n\to0$ strongly in $L^2(\partial\Omega)$ by \eqref{eq:unweighted-jump-trace-zero},
\begin{equation}\label{eq:limit-dirichlet-zero}
u|_{\partial\Omega}=0.
\end{equation}
After extracting a further subsequence, $\bq_n\rightharpoonup\bq$ weakly in $L^2(\Omega)^d$. Passing to the limit in
\[
\bq_n+\Gh(u_n,\widehat u_n)=P_{V_{h_n}}Q_n
\]
gives
\begin{equation}\label{eq:q-limit-gradient}
\bq+\nabla u=0
\qquad\text{in }L^2(\Omega)^d.
\end{equation}

Let $v\in H^1(\Omega)$ and set $(v_n,\widehat v_n)=\iota_{h_n}v=(P_Wv,P_Mv)$. The pair formulation gives
\begin{align*}
&-(\bq_n,\Gh(v_n,\widehat v_n))_{\T_{h_n}}
-\kappa_n^2(u_n,v_n)_{\T_{h_n}}
+\ii\langle\tau_n(u_n-\widehat u_n),v_n-\widehat v_n\rangle_{\partial\T_{h_n}}\\
&\qquad
+\ii\kappa_n\langle\widehat u_n,\widehat v_n\rangle_{\partial\Omega}
=
(f_n,v_n)_{\T_{h_n}}+\langle g_n,\widehat v_n\rangle_{\partial\Omega}.
\end{align*}
The right-hand side tends to zero by \eqref{eq:data-to-zero} and the boundedness of $v_n$ and $\widehat v_n$. By Lemma \ref{lem:iota-consistency}, $\Gh(v_n,\widehat v_n)\to\nabla v$ in $L^2(\Omega)^d$ and $v_n\to v$ in $L^2(\Omega)$. The stabilization term tends to zero because $\|u_n-\widehat u_n\|_{\partial\T_{h_n}}\to0$, $\tau_n\le\tau_1$, and $\|v_n-\widehat v_n\|_{\partial\T_{h_n}}\le C h_n^{1/2}\|\nabla v\|_{L^2(\Omega)}$. The boundary term tends to zero because $\widehat u_n\to0$ in $L^2(\partial\Omega)$ and $\widehat v_n$ is bounded in $L^2(\partial\Omega)$. Therefore
\[
-(\bq,\nabla v)_\Omega-\kappa_*^2(u,v)_\Omega=0
\qquad\forall v\in H^1(\Omega).
\]
Using \eqref{eq:q-limit-gradient}, we obtain
\begin{equation}\label{eq:limit-homogeneous-weak}
(\nabla u,\nabla v)_\Omega-\kappa_*^2(u,v)_\Omega=0
\qquad\forall v\in H^1(\Omega).
\end{equation}

Equation \eqref{eq:limit-homogeneous-weak} with $v\in H_0^1(\Omega)$ gives
\[
-\Delta u-\kappa_*^2u=0
\qquad\text{in }H^{-1}(\Omega).
\]
Since $u\in L^2(\Omega)$, this implies $\Delta u=-\kappa_*^2u\in L^2(\Omega)$ in the distributional sense. Let $\gamma:H^1(\Omega)\to H^{1/2}(\partial\Omega)$ denote the trace operator. The weak normal trace of $u$ is defined by
\[
\langle\partial_{\bn}u,\gamma v\rangle_{\partial\Omega}
=
(\nabla u,\nabla v)_\Omega+(\Delta u,v)_\Omega,
\qquad v\in H^1(\Omega).
\]
Using $\Delta u=-\kappa_*^2u$ and \eqref{eq:limit-homogeneous-weak}, this normal trace is zero. Together with \eqref{eq:limit-dirichlet-zero}, the zero-Cauchy-data uniqueness lemma in Appendix~\ref{app:zero-cauchy} gives $u=0$. This contradicts the strong convergence $u_n\to u$ and \eqref{eq:u-nonzero-limit}, which imply $\|u\|_{L^2(\Omega)}>0$. The contradiction proves the theorem.
\end{proof}

\subsection{Global all-mesh stability}
The next theorem combines the small-mesh compactness estimate with the coarse-mesh consequence of \cite[Theorem~3.1]{CuiZhang2014}.

\begin{theorem}\label{thm:main-stability}
Let the assumptions of \cref{thm:small-stability} hold. Then for every $h>0$ and every solution of \eqref{eq:hdg-general},
\begin{align}
&\|\bq_h\|_{\T_h}^2+
\kappa^2\|u_h\|_{\T_h}^2+
\kappa\|\widehat u_h\|_{\partial\Omega}^2+
\|\tau^{1/2}(u_h-\widehat u_h)\|_{\partial\T_h}^2
\notag\\
&\hspace{6cm}\le
C_{\rm st}
\left(
\|f\|_{\T_h}^2+
\|g\|_{\partial\Omega}^2+
\|Q\|_{\T_h}^2
\right),
\label{eq:main-stability}
\end{align}
where $C_{\rm st}$ is independent of $h$. The constant can be chosen uniformly for $\kappa\in[\kappa_0,\kappa_1]$ under \eqref{eq:tau-bounds-uniform}.
\end{theorem}

\begin{proof}
For $0<h\le h_0$, the estimate is \cref{thm:small-stability}. It remains to consider $h\ge h_0$. In this regime we use the generalized stability theorem \cite[Theorem~3.1]{CuiZhang2014}. It gives
\[
\|\bq_h\|_{\T_h}^2+\kappa^2\|u_h\|_{\T_h}^2
\le
C_{\rm CZ}(h,\kappa,\tau)
\left(
\|f\|_{\T_h}^2+
\|g\|_{\partial\Omega}^2+
\|Q\|_{\T_h}^2
\right),
\]
with $C_{\rm CZ}$ as in \eqref{eq:intro-CZ-constant}, up to constants depending only on the fixed geometric quantities. Since $h\ge h_0$, the factors $h^{-1}$ and $h^{-3}$ are bounded by $h_0^{-1}$ and $h_0^{-3}$. Since $\kappa\in[\kappa_0,\kappa_1]$ and $\tau$ satisfies \eqref{eq:tau-bounds-uniform}, this gives
\[
\|\bq_h\|_{\T_h}+\kappa\|u_h\|_{\T_h}
\le C D,
\qquad
D^2:=\|f\|_{\T_h}^2+\|g\|_{\partial\Omega}^2+\|Q\|_{\T_h}^2,
\]
where $C$ is independent of the current mesh size $h$.

It remains to control the jump and boundary trace terms. From \eqref{eq:energy-imag-general},
\[
\|\tau^{1/2}(u_h-\widehat u_h)\|_{\partial\T_h}^2+
\kappa\|\widehat u_h\|_{\partial\Omega}^2
\le
\|f\|_{\T_h}\|u_h\|_{\T_h}
+\|Q\|_{\T_h}\|\bq_h\|_{\T_h}
+\|g\|_{\partial\Omega}\|\widehat u_h\|_{\partial\Omega}.
\]
The first two terms are bounded by $CD^2$. For the last term, Young's inequality and $\kappa\ge\kappa_0$ give
\[
\|g\|_{\partial\Omega}\|\widehat u_h\|_{\partial\Omega}
\le
\frac{1}{2\kappa}\|g\|_{\partial\Omega}^2
+\frac{\kappa}{2}\|\widehat u_h\|_{\partial\Omega}^2
\le
CD^2+\frac{\kappa}{2}\|\widehat u_h\|_{\partial\Omega}^2.
\]
Absorbing the last term proves the desired estimate for $h\ge h_0$. Taking the maximum of the small-mesh and coarse-mesh constants completes the proof.
\end{proof}

\begin{corollary}\label{cor:wellposedness}
For every $h>0$, the generalized HDG system \eqref{eq:hdg-general} has a unique solution. In particular, the original HDG scheme corresponding to $Q=0$ is well posed.
\end{corollary}

\begin{proof}
The stability estimate immediately implies uniqueness for the homogeneous discrete system. Indeed, if $f=g=Q=0$, then \cref{thm:main-stability} gives $\bq_h=0$, $u_h=0$, and $\widehat u_h=0$. Since the HDG system is a square finite-dimensional linear system, uniqueness for the homogeneous system is equivalent to existence and uniqueness for every right-hand side.
\end{proof}

\section{Consequences of mesh-uniform stability}\label{sec:consequences}

\subsection{Convergence for minimal data}\label{sec:convergence}
The mesh-uniform estimate gives bounds on $u_h$, $\bq_h$, the boundary trace $\widehat u_h$, and the stabilization jump that do not deteriorate under mesh refinement. Combined with HDG compactness, these bounds allow us to pass to the limit in the discrete equations for data merely in $L^2$.

\begin{theorem}\label{thm:convergence}
Let $Q=0$, $f\in L^2(\Omega)$, and $g\in L^2(\partial\Omega)$. Let $(\bq_h,u_h,\widehat u_h)\in V_h\times W_h\times M_h$ be the unique solution of \eqref{eq:hdg-general} with $Q=0$. Then, as $h\to0$,
\begin{align}
u_h&\to u &&\text{strongly in }L^2(\Omega),\label{eq:conv-u}\\
\bq_h&\rightharpoonup -\nabla u &&\text{weakly in }L^2(\Omega)^d,\label{eq:conv-q}\\
\widehat u_h|_{\partial\Omega}&\rightharpoonup u|_{\partial\Omega} &&\text{weakly in }L^2(\partial\Omega),\label{eq:conv-trace}
\end{align}
where $u\in H^1(\Omega)$ is the unique weak solution of
\begin{equation}\label{eq:continuous-weak-robin}
(\nabla u,\nabla v)_\Omega-
\kappa^2(u,v)_\Omega+
\ii\kappa\langle u,v\rangle_{\partial\Omega}
=(f,v)_\Omega+
\langle g,v\rangle_{\partial\Omega}
\qquad\forall v\in H^1(\Omega).
\end{equation}
\end{theorem}

\begin{proof}
By \cref{thm:main-stability}, $u_h$, $\bq_h$, $\widehat u_h|_{\partial\Omega}$, and $u_h-\widehat u_h$ are bounded in the relevant norms. Since $Q=0$, \eqref{eq:q-plus-Gh} gives $\Gh(u_h,\widehat u_h)=-\bq_h$, so $\Gh(u_h,\widehat u_h)$ is also bounded. \Cref{thm:HDG-compactness} gives a subsequence and $u\in H^1(\Omega)$ satisfying \eqref{eq:conv-u} and \eqref{eq:conv-trace}, and
\[
\Gh(u_h,\widehat u_h)\rightharpoonup \nabla u.
\]
Since $\bq_h=-\Gh(u_h,\widehat u_h)$, \eqref{eq:conv-q} holds along the subsequence.

Let $v\in H^1(\Omega)$ and set $(v_h,\widehat v_h)=\iota_hv=(P_Wv,P_Mv)$. For this choice of test pair, \eqref{eq:pair-form} reads
\begin{align*}
&-(\bq_h,\Gh(P_Wv,P_Mv))_{\T_h}
-\kappa^2(u_h,P_Wv)_{\T_h}
+\ii\langle\tau(u_h-\widehat u_h),P_Wv-P_Mv\rangle_{\partial\T_h}\\
&\qquad
+\ii\kappa\langle\widehat u_h,P_Mv\rangle_{\partial\Omega}
=(f,P_Wv)_{\T_h}+\langle g,P_Mv\rangle_{\partial\Omega}.
\end{align*}
The right-hand side converges to $(f,v)_\Omega+\langle g,v\rangle_{\partial\Omega}$ by the strong convergence of $P_Wv$ and $P_Mv$. The first two left-hand side terms converge by the weak convergence of $\bq_h$, the strong convergence of $\Gh(P_Wv,P_Mv)$, the strong convergence of $u_h$, and the strong convergence of $P_Wv$. The stabilization term satisfies
\[
\left|\langle\tau(u_h-\widehat u_h),P_Wv-P_Mv\rangle_{\partial\T_h}\right|
\le
C\|u_h-\widehat u_h\|_{\partial\T_h}\,h^{1/2}\|\nabla v\|_{L^2(\Omega)}
\to0.
\]
The boundary term converges to $\ii\kappa\langle u,v\rangle_{\partial\Omega}$ by the weak convergence of $\widehat u_h|_{\partial\Omega}$ and the strong convergence of $P_Mv$ on the boundary. Therefore $u$ satisfies \eqref{eq:continuous-weak-robin}.

The continuous problem \eqref{eq:continuous-weak-robin} is uniquely solvable. For homogeneous data, taking $v=u$ and taking imaginary parts gives $\kappa\|u\|_{L^2(\partial\Omega)}^2=0$, hence $u=0$ on $\partial\Omega$. The weak formulation then gives the homogeneous Helmholtz equation in $\Omega$ and zero weak normal trace on $\partial\Omega$. The zero-Cauchy-data uniqueness lemma in Appendix~\ref{app:zero-cauchy} gives $u=0$. Thus every subsequence has the same limit, and the whole sequence converges.
\end{proof}

\subsection{Projection error estimates}\label{sec:error}

This subsection applies \cref{thm:main-stability} to the projection error equations. This is the step that removes the negative powers of $h$ from the error analysis.

Let $(\bq,u)$ be a sufficiently smooth solution of \eqref{eq:helmholtz-mixed}. Let $(\pi\bq,\Pi u)$ denote the local HDG projection used in \cite{CockburnGopalakrishnanSayas2010,CuiZhang2014}. On each element $K$, it is defined by the moment conditions
\begin{align*}
(\pi\bq-\bq,\br)_K&=0 &&\forall \br\in \Pp_{p-1}(K)^d,\\
(\Pi u-u,w)_K&=0 &&\forall w\in \Pp_{p-1}(K),\\
\langle (\pi\bq-\bq)\cdot\bn+\ii\tau(\Pi u-u),\mu\rangle_F&=0 &&\forall \mu\in\Pp_p(F),\quad F\subset\partial K.
\end{align*}
For $p=0$, the first two moment conditions are void under the convention $\Pp_{-1}=\{0\}$.

Define the projected errors
\[
\epsq=\pi\bq-\bq_h,
\qquad
\epsu=\Pi u-u_h,
\qquad
\epsuh=P_Mu-\widehat u_h,
\]
and the corresponding numerical flux error by
\[
\widehat{\bm\varepsilon}_h\cdot\bn
=
\epsq\cdot\bn+\ii\tau(\epsu-\epsuh).
\]
The following lemma shows that the projected error satisfies the same auxiliary generalized HDG system. In this system the auxiliary right-hand side $Q$ is the projection residual $\pi\bq-\bq$ in the first mixed equation.

\begin{lemma}\label{lem:error-equations}
The triple $(\epsq,\epsu,\epsuh)$ satisfies \eqref{eq:hdg-general} with
\begin{equation}\label{eq:error-data}
Q=\pi\bq-\bq,
\qquad
f=-\kappa^2(\Pi u-u),
\qquad
g=0.
\end{equation}
\end{lemma}

\begin{proof}
The exact solution satisfies, for all $(\btau_h,v_h)\in V_h\times W_h$,
\begin{align}
(\bq,\btau_h)_{\T_h}-(u,\nabla\cdot\btau_h)_{\T_h}
+\langle u,\btau_h\cdot\bn\rangle_{\partial\T_h}&=0,\label{eq:exact-first-projected-proof}\\
-(\bq,\nabla v_h)_{\T_h}-\kappa^2(u,v_h)_{\T_h}
+\langle\bq\cdot\bn,v_h\rangle_{\partial\T_h}&=(f,v_h)_{\T_h}.
\label{eq:exact-second-projected-proof}
\end{align}
Using the projection conditions and the definition of $P_M$, \eqref{eq:exact-first-projected-proof} gives
\[
(\pi\bq,\btau_h)_{\T_h}-(\Pi u,\nabla\cdot\btau_h)_{\T_h}
+\langle P_Mu,\btau_h\cdot\bn\rangle_{\partial\T_h}
=(\pi\bq-\bq,\btau_h)_{\T_h}.
\]
Subtracting \eqref{eq:hdg-general-a} with $Q=0$ gives
\[
(\epsq,\btau_h)_{\T_h}-(\epsu,\nabla\cdot\btau_h)_{\T_h}
+\langle\epsuh,\btau_h\cdot\bn\rangle_{\partial\T_h}
=(\pi\bq-\bq,\btau_h)_{\T_h}.
\]
This is the first equation of the auxiliary system for the error, with $Q=\pi\bq-\bq$.

For the second equation, the projection relation on each face gives
\[
\langle(\pi\bq-\bq)\cdot\bn+\ii\tau(\Pi u-u),v_h\rangle_{\partial\T_h}=0,
\]
because $v_h|_{\partial K}$ is a polynomial of degree at most $p$ on each face. Hence \eqref{eq:exact-second-projected-proof} can be written as
\[
-(\pi\bq,\nabla v_h)_{\T_h}
-\kappa^2(\Pi u,v_h)_{\T_h}
+\langle \pi\bq\cdot\bn+\ii\tau(\Pi u-u),v_h\rangle_{\partial\T_h}
=
(f,v_h)_{\T_h}
-\kappa^2(\Pi u-u,v_h)_{\T_h}.
\]
Subtracting the HDG equation \eqref{eq:hdg-general-b} with $Q=0$ gives
\[
-(\epsq,\nabla v_h)_{\T_h}
-\kappa^2(\epsu,v_h)_{\T_h}
+\langle\widehat{\bm\varepsilon}_h\cdot\bn,v_h\rangle_{\partial\T_h}
=
-\kappa^2(\Pi u-u,v_h)_{\T_h},
\]
where we used the definition of $\widehat{\bm\varepsilon}_h\cdot\bn$ and the identity
\[
\langle\tau(P_Mu-u),v_h\rangle_{\partial\T_h}=0.
\]
The latter follows from the facewise constancy of $\tau$ and the $L^2$ projection property of $P_M$.

It remains to verify the boundary and interior flux equations. On $\partial\Omega$, the continuous boundary condition gives
\[
-\bq\cdot\bn+\ii\kappa u=g.
\]
Using the projection relation with test functions on boundary faces and the identity
\[
\langle\tau(P_Mu-u),\mu_h\rangle_{\partial\Omega}=0,
\]
we obtain
\[
\langle-\widehat{\bm\varepsilon}_h\cdot\bn+\ii\kappa\epsuh,\mu_h\rangle_{\partial\Omega}=0
\qquad\forall \mu_h\in M_h.
\]
On interior faces, the exact normal flux is single-valued with opposite signs from the two neighboring elements, while the HDG numerical flux satisfies the conservation equation. Using again the facewise projection relation and $\langle\tau(P_Mu-u),\mu_h\rangle=0$, we get
\[
\langle\widehat{\bm\varepsilon}_h\cdot\bn,\mu_h\rangle_{\partial\T_h\setminus\partial\Omega}=0
\qquad\forall \mu_h\in M_h.
\]
Therefore $(\epsq,\epsu,\epsuh)$ satisfies \eqref{eq:hdg-general} with the data in \eqref{eq:error-data}.
\end{proof}

\begin{theorem}\label{thm:projection-error}
Under the assumptions of \cref{thm:main-stability}, the projected errors satisfy
\begin{align}
&\|\pi\bq-\bq_h\|_{\T_h}^2+
\kappa^2\|\Pi u-u_h\|_{\T_h}^2+
\kappa\|P_Mu-\widehat u_h\|_{\partial\Omega}^2
\notag\\
&\qquad+
\|\tau^{1/2}\big((\Pi u-u_h)-(P_Mu-\widehat u_h)\big)\|_{\partial\T_h}^2
\notag\\
&\hspace{5cm}
\le
C_{\rm st}
\left(
\|\pi\bq-\bq\|_{\T_h}^2+
\kappa^4\|\Pi u-u\|_{\T_h}^2
\right),
\label{eq:projection-error-estimate}
\end{align}
where $C_{\rm st}$ is independent of $h$. The same local uniformity in $\kappa\in[\kappa_0,\kappa_1]$ as in \cref{thm:main-stability} holds.
\end{theorem}

\begin{proof}
Apply \cref{thm:main-stability} to the error equations in Lemma~\ref{lem:error-equations}. Since the error system has data \eqref{eq:error-data}, the right-hand side is
\[
\|\pi\bq-\bq\|_{\T_h}^2+\kappa^4\|\Pi u-u\|_{\T_h}^2.
\]
This gives exactly \eqref{eq:projection-error-estimate}.
\end{proof}

\begin{corollary}\label{cor:projection-defect-error}
Let
\[
\eta_q:=\|\bq-\pi\bq\|_{\T_h},
\qquad
\eta_u:=\|u-\Pi u\|_{\T_h}.
\]
Then
\begin{equation}\label{eq:projection-defect-error}
\|\bq-\bq_h\|_{\T_h}^2+
\kappa^2\|u-u_h\|_{\T_h}^2
\le
C\left[
\eta_q^2+
\kappa^2\eta_u^2+
C_{\rm st}\left(\eta_q^2+\kappa^4\eta_u^2\right)
\right],
\end{equation}
where $C$ is independent of $h$.
\end{corollary}

\begin{proof}
The triangle inequalities
\[
\|\bq-\bq_h\|_{\T_h}\le \eta_q+\|\pi\bq-\bq_h\|_{\T_h},
\qquad
\|u-u_h\|_{\T_h}\le \eta_u+\|\Pi u-u_h\|_{\T_h}
\]
combined with \cref{thm:projection-error} give the estimate.
\end{proof}

The projection-defect estimate above is the main error consequence of the mesh-uniform stability theorem. Explicit error bounds in the notation of \cite[Corollary~3.2]{CuiZhang2014} follow by combining Corollary \ref{cor:projection-defect-error} with the corresponding local HDG projection approximation estimates. We record this consequence in Appendix~\ref{app:explicit-CZ-error}, where the hypotheses are stated explicitly and the algebraic proof is given.

\begin{remark}\label{rem:error-meaning}
Theorem~\ref{thm:projection-error} has the same algebraic role as Theorem~3.3 of \cite{CuiZhang2014}: it applies a generalized stability estimate to the projected error equations. The explicit Cui--Zhang-type formulas in Appendix~\ref{app:explicit-CZ-error} correspond to \cite[Corollary~3.2]{CuiZhang2014}. The improvement is the replacement of the mesh-singular stability factor by a constant independent of $h$. The constants are locally uniform for $\kappa$ in a prescribed compact subinterval of $(0,\infty)$ and may depend on the fixed polynomial degree $p$.
\end{remark}

\section{Numerical examples}\label{sec:numerics}

This section reports numerical experiments for the HDG method
\eqref{eq:hdg-general}--\eqref{eq:numerical-flux}. The purpose of the experiments is to test the main assertion of the analysis: for fixed polynomial degree and for
stabilizations bounded above and below independently of the mesh size, the
observable discrete stability does not exhibit the negative powers of $h$ appearing
in the explicit stability factor \eqref{eq:intro-CZ-constant} given in \cite{CuiZhang2014}.

The code used for all numerical examples presented in this paper is available on GitHub at
\url{https://github.com/yyue24/HDG-Helmholtz-numerics}. All computations were performed with complex arithmetic in double precision. The
global linear systems were statically condensed to the skeletal trace unknowns and
then solved by a sparse direct solver. In two dimensions we use the uniform
triangulation of $\Omega=(0,1)^2$ obtained by dividing each square cell into two
triangles; in three dimensions we use the analogous uniform tetrahedralization of
$\Omega=(0,1)^3$. In all tables below, $h_N=1/N$ denotes the mesh parameter. The
maximal simplex diameter differs from $h_N$ only by the fixed factors $\sqrt2$ in
two dimensions and $\sqrt3$ in three dimensions.

We report three residual diagnostics. The algebraic residual is
\[
\frac{\|Ax-b\|_2}{\|b\|_2},
\]
for the condensed linear system. The internal flux residual is the weak face
residual
\[
C_{\rm int}^2
=
\sum_{F\in\E_h^i} (r_F^{\rm int})^*M_F^{-1}r_F^{\rm int},
\qquad
r_F^{\rm int}(\mu_h)
=
\left\langle
\widehat{\bq}_h^+\cdot\bn^+
+
\widehat{\bq}_h^-\cdot\bn^-,
\mu_h
\right\rangle_F,
\]
where $M_F$ is the face mass matrix. The boundary residual is the weak Robin
residual
\[
C_{\rm bdry}^2
=
\sum_{F\subset\partial\Omega} (r_F^{\rm bdry})^*M_F^{-1}r_F^{\rm bdry},
\qquad
r_F^{\rm bdry}(\mu_h)
=
\left\langle
-\widehat{\bq}_h\cdot\bn+\ii\kappa\widehat u_h-g,
\mu_h
\right\rangle_F.
\]
The residuals reported below are weak residuals in the discrete test space. For
non-polynomial data, such as plane waves, pointwise boundary residuals are not used
as a pass--fail criterion.

For the stability experiments we use the following discrete energy quantity, which we refer to below as the energy norm:
\begin{equation}\label{eq:num-energy-ratio}
E_h^2
=
\|\bq_h\|_{\T_h}^2
+
\kappa^2\|u_h\|_{\T_h}^2
+
\kappa\|\widehat u_h\|_{\partial\Omega}^2
+
\|\tau^{1/2}(u_h-\widehat u_h)\|_{\partial\T_h}^2
\end{equation}
and
\[
D_h^2
=
\|f_h\|_{\T_h}^2+\|g_h\|_{\partial\Omega}^2+\|Q_h\|_{\T_h}^2,
\qquad
R_h=\frac{E_h}{D_h}.
\]
The random right-hand sides used in the stability tests are mass-whitened complex
Gaussian samples and then normalized so that $D_h=1$. This avoids a sampling bias
caused by the scaling of the polynomial basis.

To compare the measured stability ratio with the old explicit stability factor, we
plot the relative Cui--Zhang reference curve
\begin{equation}\label{eq:num-relative-CZ}
B_{\rm CZ}^{\rm rel}(h)
=
R_{h_0}^{\max}
\left(
\frac{
C_{\rm CZ}(h,\kappa,\tau)
}{
C_{\rm CZ}(h_0,\kappa,\tau)
}
\right)^{1/2},
\end{equation}
where $h_0$ is the coarsest mesh in the test and $C_{\rm CZ}$ is the expression in
\eqref{eq:intro-CZ-constant}. Thus \eqref{eq:num-relative-CZ} compares only the
mesh-dependent growth rate and does not assume that the hidden constants in the
old bound are sharp.

\subsection{Plane-wave convergence}\label{sec:num-plane-wave}

We first verify the implementation using the exact two-dimensional plane wave
\[u(x)=\exp(\ii\kappa d\cdot x),
\qquad
d=\frac{(1,2)}{\sqrt5},
\qquad
\kappa=5.
\]
Then $f=0$,
\[
\bq=-\nabla u=-\ii\kappa d\,u,
\qquad
g=\partial_{\bn}u+\ii\kappa u
=
\ii\kappa(d\cdot\bn+1)u.
\]
We test polynomial degrees $p=0,1,2$ and two stabilizations, $\tau=1$ and
$\tau=\kappa$. All element and face integrals are evaluated with quadrature order
\[
q_{\rm quad}(p)=2p+12.
\]
An additional quadrature audit with order $q_{\rm quad}(p)+4$ changes the main
errors by less than $5\times10^{-3}$ in relative value.

\Cref{tab:num-exp1-final} reports the finest-grid errors and the final observed
orders. The observed rates are approximately $p+1$ for both $u_h$ and $\bq_h$.
The largest algebraic residual over all rows of this experiment is
$1.58\times10^{-11}$, the largest weak internal flux residual is
$1.00\times10^{-10}$, and the largest weak Robin residual is
$1.67\times10^{-11}$.

\begin{table}[tbp]
\centering
\caption{Two-dimensional plane-wave convergence at the finest grid $N=128$,
$\kappa=5$. The rates are the observed rates between $N=64$ and $N=128$.}
\label{tab:num-exp1-final}
\begin{tabular}{cccccc}
\toprule
$p$ & $\tau$ &
$\|u-u_h\|_{\T_h}$ & rate &
$\|\bq-\bq_h\|_{\T_h}$ & rate \\
\midrule
0 & $1$        & $4.83\times10^{-2}$ & 0.956 & $2.00\times10^{-1}$ & 0.957 \\
0 & $\kappa$  & $2.13\times10^{-2}$ & 0.980 & $1.14\times10^{-1}$ & 0.980 \\
1 & $1$        & $1.05\times10^{-4}$ & 2.000 & $1.75\times10^{-4}$ & 2.002 \\
1 & $\kappa$  & $3.38\times10^{-5}$ & 2.001 & $2.33\times10^{-4}$ & 2.001 \\
2 & $1$        & $2.61\times10^{-7}$ & 3.000 & $4.39\times10^{-7}$ & 3.000 \\
2 & $\kappa$  & $8.52\times10^{-8}$ & 3.000 & $5.85\times10^{-7}$ & 3.000 \\
\bottomrule
\end{tabular}
\end{table}

\subsection{Mesh-uniform stability}\label{sec:num-mesh-uniform}

The central numerical test concerns the empirical mesh-uniform stability ratio
$R_h$ in \eqref{eq:num-energy-ratio}. We set $\kappa=5$, $\tau=1$, and test
$p=0,1,2$ on meshes $N=8,16,32,64,128$. For each $(p,N)$ we generate 20 independent
mass-whitened complex random right-hand sides with $Q_h$, $f_h$, and $g_h$ all
present and normalized so that $D_h=1$. We record the maximum and median values
of $R_h$.

The results are shown in \Cref{fig:num-stability-cz} and summarized in
\Cref{tab:num-exp2-summary}. In all three polynomial degrees, $R_h^{\max}$ remains
bounded as the mesh is refined. In fact, it decreases or remains essentially flat.
By contrast, the relative Cui--Zhang reference curve $B_{\rm CZ}^{\rm rel}$ grows by
roughly four orders of magnitude on the same meshes. For example, for $p=1$,
$R_h^{\max}$ changes from $0.545$ at $N=8$ to $0.472$ at $N=128$, while
$B_{\rm CZ}^{\rm rel}$ changes from $0.545$ to $2078$.

\begin{figure}[tbp]
\centering
\includegraphics[width=.74\linewidth]{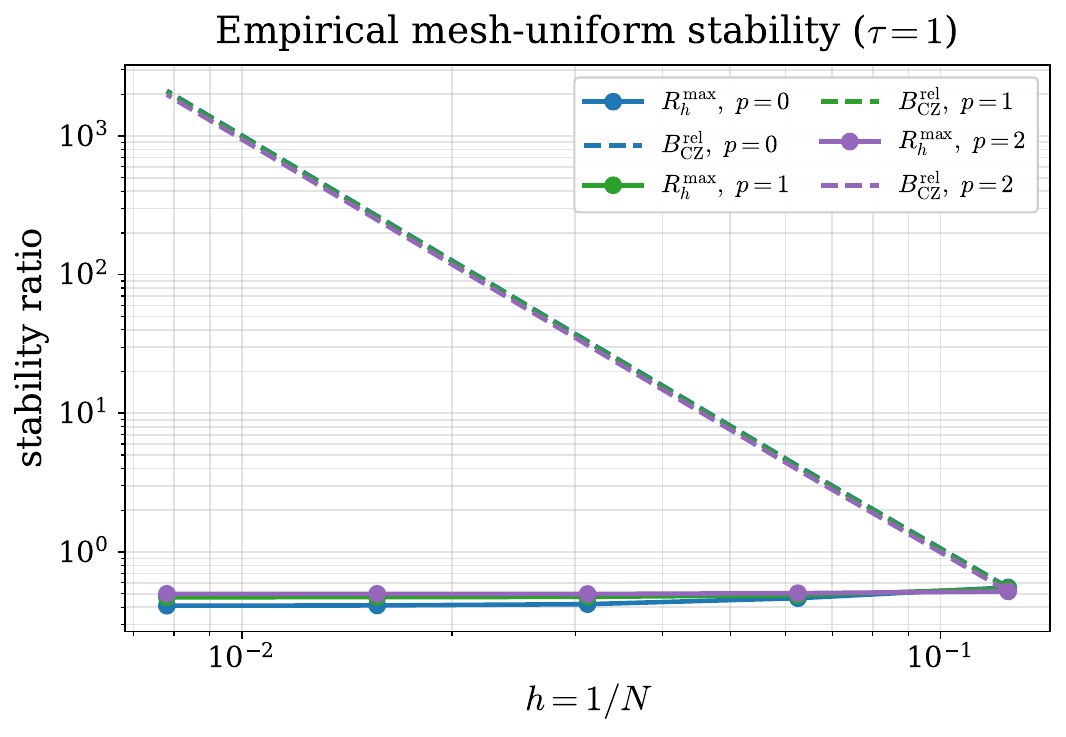}
\caption{Empirical mesh-uniform stability ratio $R_h^{\max}$
for the generalized HDG system, compared with the relative Cui--Zhang reference
curve \eqref{eq:num-relative-CZ}. The computed stability ratios stay bounded,
whereas the old reference curve reflects the $h^{-3}$ growth in
\eqref{eq:intro-CZ-constant}.}
\label{fig:num-stability-cz}
\end{figure}

\begin{table}[tbp]
\centering
\caption{Mesh-uniform stability test for $\kappa=5$, $\tau=1$, and 20
mass-whitened random right-hand sides per mesh.}
\label{tab:num-exp2-summary}
\begin{tabular}{ccccc}
\toprule
$p$ & $R_{\max}(N=8)$ & $R_{\max}(N=128)$ &
slope of $\log R_{\max}$ vs.\ $\log h$ &
$B_{\rm CZ}^{\rm rel}(N=128)$ \\
\midrule
0 & 0.555 & 0.410 & 0.105 & 2117 \\
1 & 0.545 & 0.472 & 0.046 & 2078 \\
2 & 0.520 & 0.500 & 0.013 & 1982 \\
\bottomrule
\end{tabular}
\end{table}

These results are consistent with the mesh-uniform estimate proved in
\Cref{sec:stability}. They also indicate that the negative powers of $h$ in
\eqref{eq:intro-CZ-constant} are not observed in the actual discrete resolvent
measured by \eqref{eq:num-energy-ratio}.

\subsection{The generalized system and the \texorpdfstring{$Q$}{Q}-forcing term}\label{sec:num-q-forcing}

The $Q$ term in \eqref{eq:hdg-general-a} is not present in the physical Helmholtz
solve, but it is essential in the projected error equations. We therefore test the
generalized system in three regimes:
\[
\text{all: } Q_h,f_h,g_h\ne0,
\qquad
\text{fg\_only: } Q_h=0,\ f_h,g_h\ne0,
\qquad
\text{q\_only: } Q_h\ne0,\ f_h=g_h=0.
\]
As before, every sample is mass-whitened and normalized by $D_h=1$. We use
$\kappa=5$, $\tau=1$, $p=0,1,2$, $N=8,16,32,64,128$, and 20 samples per parameter
set.

\Cref{fig:num-qforcing-cz-all,fig:num-qforcing-cz-fg-only,fig:num-qforcing-cz-q-only}
and \Cref{tab:num-exp3-summary} show that the empirical stability ratio remains
bounded under mesh refinement in all three forcing configurations. In particular,
the isolated $Q$-forcing case stays uniformly bounded, which is important because
the proof of the error estimate in \Cref{sec:error} uses the generalized stability
estimate with $Q=\pi\bq-\bq$.

\begin{figure}[tbp]
\centering
\includegraphics[width=.74\linewidth]{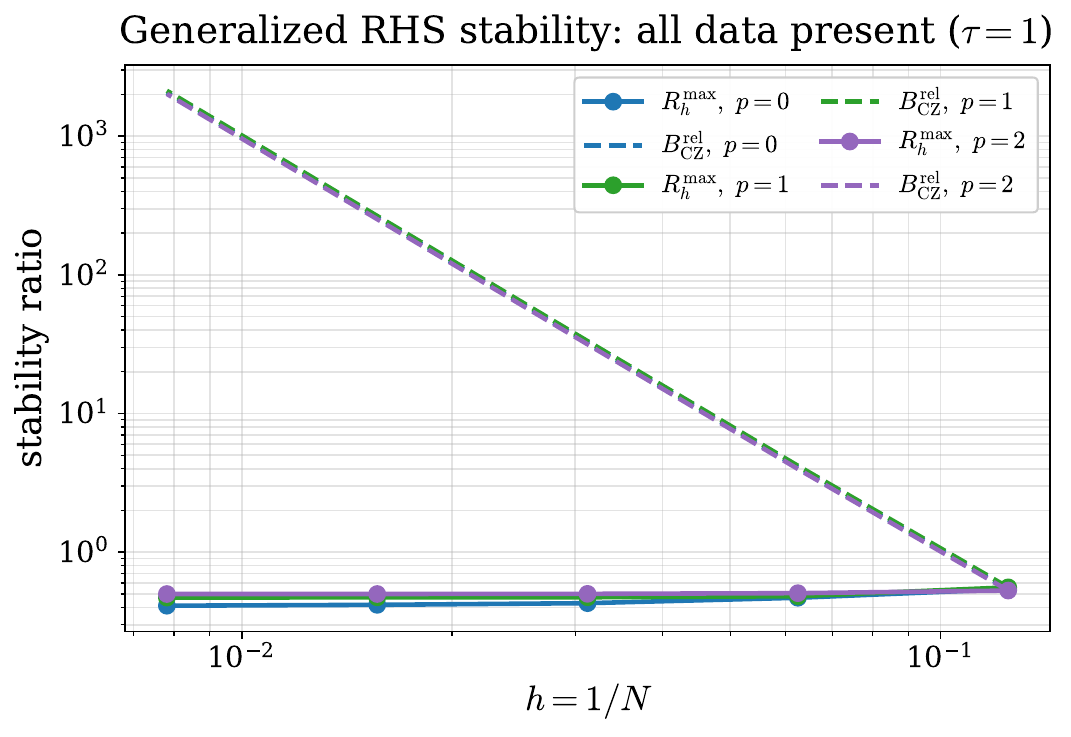}
\caption{Generalized stability test with all data present. The empirical stability ratio $R_h^{\max}$ remains uniformly bounded for $p=0,1,2$ as the mesh is refined, while the relative Cui--Zhang reference curve $B_{\rm CZ}^{\rm rel}$ exhibits the mesh-dependent growth predicted by the old explicit bound.}
\label{fig:num-qforcing-cz-all}
\end{figure}

\begin{figure}[tbp]
\centering
\includegraphics[width=.74\linewidth]{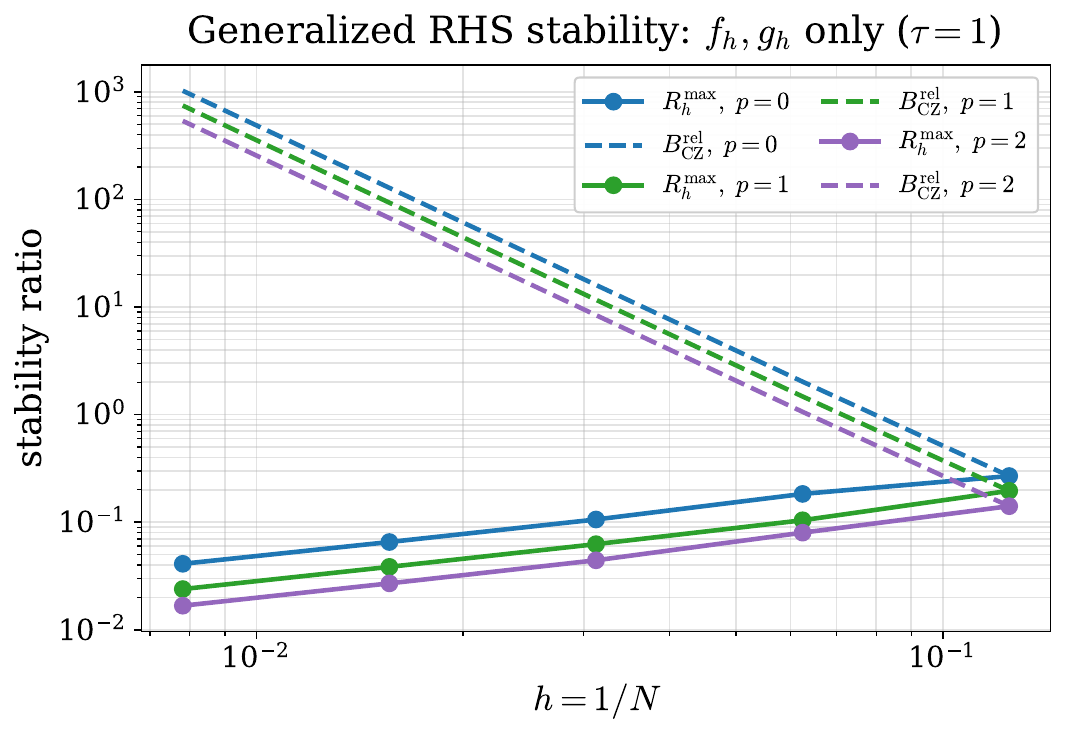}
\caption{Generalized stability test with only $f_h$ and $g_h$ present. In this case the measured stability ratio decreases under mesh refinement, whereas the relative Cui--Zhang reference curve still grows rapidly.}
\label{fig:num-qforcing-cz-fg-only}
\end{figure}

\begin{figure}[tbp]
\centering
\includegraphics[width=.74\linewidth]{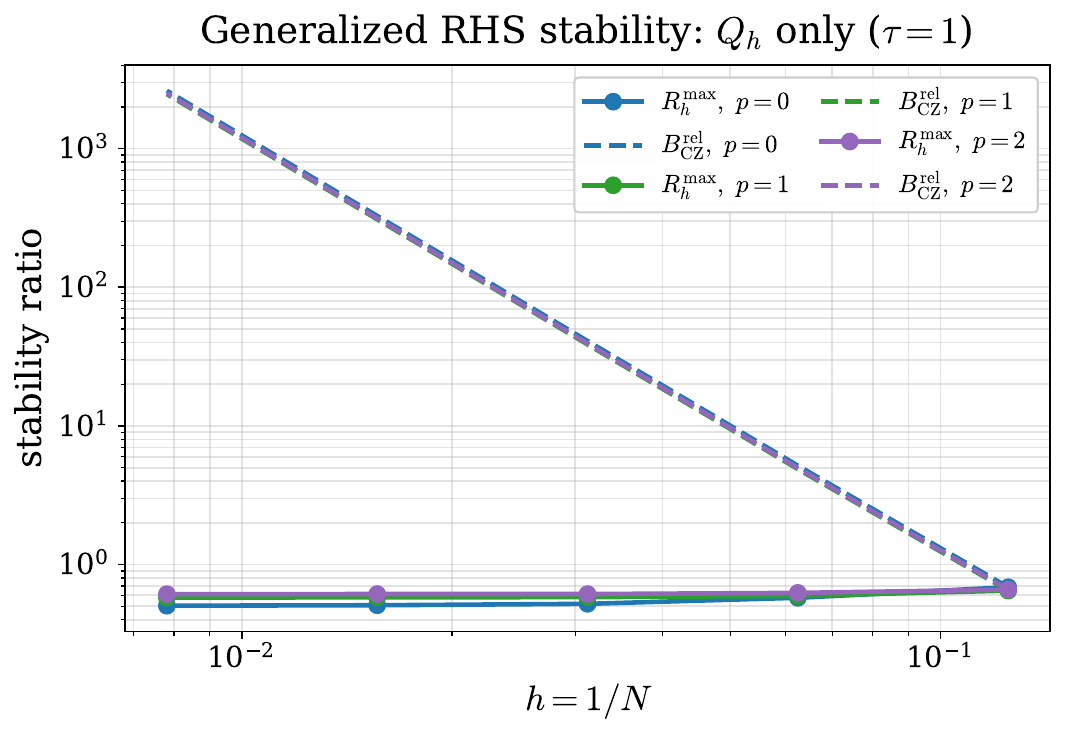}
\caption{Generalized stability test with only $Q_h$ present. The $Q$-only case remains uniformly bounded under mesh refinement, confirming that the auxiliary term used in the error equations does not introduce an observable mesh-size blow-up.}
\label{fig:num-qforcing-cz-q-only}
\end{figure}

\begin{table}[tbp]
\centering
\caption{Generalized stability test for the three forcing regimes. The table
reports $R_{\max}$ at the coarsest and finest meshes and the slope of
$\log R_{\max}$ against $\log h$.}
\label{tab:num-exp3-summary}
\begin{tabular}{ccccc}
\toprule
case & $p$ & $R_{\max}(N=8)$ & $R_{\max}(N=128)$ & slope \\
\midrule
all      & 0 & 0.542 & 0.411 & 0.096 \\
all      & 1 & 0.557 & 0.470 & 0.052 \\
all      & 2 & 0.529 & 0.499 & 0.018 \\
fg\_only & 0 & 0.269 & 0.041 & 0.691 \\
fg\_only & 1 & 0.196 & 0.024 & 0.751 \\
fg\_only & 2 & 0.141 & 0.017 & 0.771 \\
q\_only  & 0 & 0.686 & 0.505 & 0.106 \\
q\_only  & 1 & 0.649 & 0.577 & 0.037 \\
q\_only  & 2 & 0.659 & 0.612 & 0.024 \\
\bottomrule
\end{tabular}
\end{table}

\subsection{Wave-number and stabilization dependence}\label{sec:num-kappa}

The theory in this paper gives local uniformity on every compact interval
$0<\kappa_0\le\kappa\le\kappa_1<\infty$. It does not assert a high-frequency
uniform estimate as $\kappa\to\infty$. To illustrate this distinction, we run a
wave-number sweep with
\[
\kappa\in\{1,2,4,8,16,32\},
\qquad
\tau\in\{1,\kappa,0.1\kappa,10\kappa\}.
\]
We use the same exact plane-wave family $u_\kappa(x)=\exp(\ii\kappa d\cdot x)$ with $d=(1,2)/\sqrt5$ and set $f$ and $g$ accordingly. We test both a fixed mesh $N=128$ and a fixed-resolution rule based on points per wavelength. Since the wavelength is $\lambda=2\pi/\kappa$ and $h=1/N$, the number of grid intervals per wavelength is $\lambda/h=2\pi N/\kappa$. In the fixed-resolution runs, we use the representative choice of at least $12$ grid intervals per wavelength, choosing $N$ as the smallest dyadic integer with $N\ge8$ and $2\pi N/\kappa\ge12$.

The error behavior reflects the usual resolution dependence of Helmholtz
discretizations. On the fixed mesh with $\tau=1$, the $L^2$ error in $u$ increases
from approximately $1.1\times10^{-6}$ at $\kappa=1$ to
$2.48\times10^{-2}$ at $\kappa=32$. Under the fixed-points-per-wavelength rule,
the high-frequency error is still visible. However, the stabilization choice matters:
at $\kappa=32$, $\tau=\kappa$ reduces the scalar error from
$8.71\times10^{-2}$ for $\tau=1$ to $9.02\times10^{-3}$.

\begin{table}[tbp]
\centering
\caption{Representative wave-number sweep results at $\kappa=32$ under the
fixed-points-per-wavelength rule. Here $N=64$ and the number of points per
wavelength is approximately $12.57$. The last column is the maximum random
stability ratio over 10 mass-whitened random samples.}
\label{tab:num-kappa32}
\begin{tabular}{cccc}
\toprule
$\tau$ & $\|u-u_h\|_{\T_h}$ & $\|\bq-\bq_h\|_{\T_h}$ & random $R_{\max}$ \\
\midrule
$1$          & $8.71\times10^{-2}$ & $1.78$ & 0.582 \\
$\kappa$    & $9.02\times10^{-3}$ & $3.33\times10^{-1}$ & 0.610 \\
$0.1\kappa$ & $4.04\times10^{-2}$ & $8.44\times10^{-1}$ & 0.598 \\
$10\kappa$  & $3.77\times10^{-2}$ & $1.82$ & 0.786 \\
\bottomrule
\end{tabular}
\end{table}

Thus the experiments support the local-in-$\kappa$ interpretation of the theorem
but should not be read as evidence for a pollution-free or high-frequency uniform
method. Additional data for this wave-number sweep are reported in
\Cref{app:num-kappa}.

\subsection{Nonsmooth \texorpdfstring{$L^2$}{L2} forcing}\label{sec:num-rough}

To test the minimal-data narrative, we solve a problem with rough $L^2$ forcing on
$\Omega=(0,1)^2$:
\[
f=\chi_{D_1}+\ii\chi_{D_2},
\qquad
g=0,
\]
where
\[
D_1=[1/4,3/4]\times[1/4,3/4],
\qquad
D_2=[1/8,5/8]\times[3/8,7/8].
\]
The rectangles are aligned with the dyadic meshes used in the experiment. Since no
closed-form exact solution is used, we compare against a reference solution computed
with $N_{\rm ref}=256$ and $p_{\rm ref}=2$. The reference errors decrease
monotonically: for $p=1$, the scalar reference error decreases from
$1.55\times10^{-3}$ at $N=8$ to $5.97\times10^{-6}$ at $N=128$, with observed
order about $2$. The energy norm remains bounded, approximately between $0.25$
and $0.33$. Detailed results for this nonsmooth forcing test are given in
\Cref{app:num-rough}.

\subsection{Three-dimensional verification}\label{sec:num-3d}

We also verify the three-dimensional implementation using
\[
\Omega=(0,1)^3,
\qquad
u(x,y,z)=\exp\!\left(\ii\kappa\frac{x+y+z}{\sqrt3}\right),
\qquad
\kappa=3.
\]
For $p=1$, the observed scalar convergence orders on the final mesh refinement are
$2.004$ for $\tau=1$ and $2.016$ for $\tau=\kappa$. The corresponding flux orders
are $2.011$ and $2.031$. Detailed results for the three-dimensional experiment are
given in \Cref{app:num-3d}.

Overall, the numerical evidence is consistent with the conclusions of the paper:
the observed HDG stability ratio is mesh-uniform for fixed $p$ and fixed finite
wave-number ranges, while the old explicit Cui--Zhang reference curve displays the
negative mesh powers inherited from the previous proof technique.

\section{Concluding remarks}

We have proved an $h$-uniform generalized stability estimate for the HDG Helmholtz method analyzed in \cite{CuiZhang2014}, within the same geometric and mesh framework, for fixed polynomial degree and uniformly bounded stabilization parameters. The proof replaces the Rellich-inverse-estimate route by an HDG compactness argument based on the discrete distributional gradient and the Crouzeix--Raviart lifting mechanism. Applying the generalized estimate to the projection error equations removes the negative powers of $h$ from the stability factor in the error analysis. The numerical experiments in \Cref{sec:numerics} support the same conclusion: the measured stability ratios remain bounded under mesh refinement, while the old relative Cui--Zhang reference curve exhibits the mesh-singular growth predicted by the previous explicit bound. Thus, for the HDG setting considered here, the mesh-size blow-up in the previous stability and error bounds is not an intrinsic feature of the discrete Helmholtz operator.

The compactness input used in this paper is a fixed-degree HDG compactness result. A simultaneous $h$- and $p$-uniform stability estimate would require a $p$-robust compactness theorem for the HDG discrete distributional gradient and the Crouzeix--Raviart lifting mechanism, avoiding polynomial-degree dependent inverse or trace constants. Such a result is not proved here. Another important open point is to derive a usable explicit dependence of the stability constant on $\kappa$ as $\kappa\to\infty$; the present theorem is only locally uniform on prescribed compact wave-number intervals. Establishing these two refinements remains a natural direction for future work.

\appendix

\section{The CR lifting mechanism behind compactness}\label{app:compactness}

This appendix gives the details behind the compactness statement used in \cref{thm:HDG-compactness}. The argument is a specialized form of the HDG compactness framework of \cite{JiangWalkingtonYue2025}; it is included to make clear how the Crouzeix--Raviart lifting interacts with the HDG discrete distributional gradient.

Let $\widehat w_h\in M_h$ and define its face average by
\[
\overline{\widehat w_h}|_F=\frac1{|F|}\int_F\widehat w_h\,ds.
\]
The Crouzeix--Raviart lifting $\Lh\widehat w_h\in\CR(\T_h)$ is the function whose face degrees of freedom are these averages. Equivalently, $\Lh\widehat w_h$ is the CR lifting of the facewise constant function $\overline{\widehat w_h}$, not of the full polynomial trace $\widehat w_h$.

The first key estimate is the projection property
\begin{equation}\label{eq:appendix-gradient-projection}
\|\nabla_h\Lh\widehat w_h\|_{L^2(\Omega)}
\le
\|\Gh(w_h,\widehat w_h)\|_{L^2(\Omega)}.
\end{equation}
Indeed, if $\br_0$ is piecewise constant, then $\nabla\cdot\br_0=0$ on each element, and the definition of $\Gh$ gives
\[
(\Gh(w_h,\widehat w_h),\br_0)_{\T_h}
=(\nabla_h w_h,\br_0)_{\T_h}
+\langle \widehat w_h-w_h,\br_0\cdot\bn\rangle_{\partial\T_h}
=\langle \widehat w_h,\br_0\cdot\bn\rangle_{\partial\T_h}.
\]
Since $\br_0\cdot\bn$ is constant on each face, replacing $\widehat w_h$ by its face average does not change the last pairing. Integration by parts for the CR function gives
\[
\langle \widehat w_h,\br_0\cdot\bn\rangle_{\partial\T_h}
=
\langle \Lh\widehat w_h,\br_0\cdot\bn\rangle_{\partial\T_h}
=(\nabla_h\Lh\widehat w_h,\br_0)_{\T_h}.
\]
Thus $\nabla_h\Lh\widehat w_h$ is the $L^2$ projection of $\Gh(w_h,\widehat w_h)$ onto the piecewise constant vector fields, which proves \eqref{eq:appendix-gradient-projection}.

The second key estimate is the local difference bound
\begin{equation}\label{eq:appendix-difference}
\|w_h-\Lh\widehat w_h\|_{L^2(K)}
\le
C\bigl(h_K\|\Gh(w_h,\widehat w_h)\|_{L^2(K)}+h_K^{1/2}\|w_h-\widehat w_h\|_{L^2(\partial K)}\bigr).
\end{equation}
For completeness we recall the proof. Testing the definition of $\Gh$ on $K$ with $\Gh(w_h,\widehat w_h)-\nabla w_h$ gives, using the inverse trace inequality,
\[
h_K\|\Gh(w_h,\widehat w_h)-\nabla w_h\|_{L^2(K)}^2
\le
C\|w_h-\widehat w_h\|_{L^2(\partial K)}^2.
\]
Hence
\[
h_K\|\nabla w_h\|_{L^2(K)}^2
\le
C\left(h_K\|\Gh(w_h,\widehat w_h)\|_{L^2(K)}^2+
\|w_h-\widehat w_h\|_{L^2(\partial K)}^2\right).
\]
Together with \eqref{eq:appendix-gradient-projection}, a local Poincar\'e inequality around the boundary average of $w_h-\Lh\widehat w_h$ yields the derivative part of \eqref{eq:appendix-difference}. The boundary-average part is controlled by
\[
\left\|(w_h-\Lh\widehat w_h)_{\partial K}\right\|_{L^2(K)}
\le
C h_K^{1/2}\|w_h-\widehat w_h\|_{L^2(\partial K)},
\]
because the face averages of $\Lh\widehat w_h$ agree with the face averages of $\widehat w_h$. Combining these two bounds gives \eqref{eq:appendix-difference}.

Now assume \eqref{eq:compactness-bound} and let $h\to0$. Summing \eqref{eq:appendix-difference} over all elements gives
\[
\|w_h-\Lh\widehat w_h\|_{L^2(\Omega)}\to0,
\]
while \eqref{eq:appendix-gradient-projection} gives a uniform broken $H^1$ bound for the CR sequence $\Lh\widehat w_h$. The compactness of the CR space therefore gives a subsequence and a function $w\in H^1(\Omega)$ such that
\[
\Lh\widehat w_h\to w\quad\text{strongly in }L^2(\Omega),
\qquad
\nabla_h\Lh\widehat w_h\rightharpoonup\nabla w\quad\text{weakly in }L^2(\Omega)^d.
\]
The vanishing difference above then implies $w_h\to w$ strongly in $L^2(\Omega)$. To identify the weak limit of $\Gh(w_h,\widehat w_h)$, test the defining identity of $\Gh$ against smooth vector fields projected into $V_h$ and pass to the limit; the volume convergence of $w_h$ and the cancellation of interior face contributions show that the weak limit is $\nabla w$. The same argument applied to test fields with nonzero normal trace on the boundary identifies the weak boundary limit of $\widehat w_h|_{\partial\Omega}$ with $w|_{\partial\Omega}$. This proves the compactness statement used in \cref{thm:HDG-compactness}.

\section{Zero Cauchy data uniqueness}\label{app:zero-cauchy}

\begin{lemma}\label{lem:zero-cauchy}
Let $\Omega\subset\R^d$, $d=2,3$, be a connected bounded Lipschitz polyhedral domain. Suppose $u\in H^1(\Omega)$ satisfies
\[
-\Delta u-\kappa^2u=0\quad\text{in }\Omega,
\qquad
u|_{\partial\Omega}=0,
\qquad
\partial_{\bn}u|_{\partial\Omega}=0
\]
in the trace and weak normal-trace sense. Then $u=0$ in $\Omega$.
\end{lemma}

\begin{proof}
Since $\Delta u=-\kappa^2u\in L^2(\Omega)$, we have
\[
u\in H_\Delta(\Omega):=\{w\in H^1(\Omega):\Delta w\in L^2(\Omega)\}.
\]
For functions in $H_\Delta(\Omega)$ the weak normal trace is defined by Green's identity:
\[
\langle\partial_{\bn}u,\gamma\varphi\rangle_{\partial\Omega}
=
(\nabla u,\nabla\varphi)_\Omega+(\Delta u,\varphi)_\Omega
\qquad\forall \varphi\in H^1(\Omega),
\]
where $\gamma:H^1(\Omega)\to H^{1/2}(\partial\Omega)$ is the trace operator.

Choose a relatively open subset $F$ of one flat boundary face of the polyhedral domain. After a rigid change of coordinates, assume that the supporting hyperplane of $F$ is $\{x_d=0\}$ and that, in a ball $B$ centered at a point of $F$, the domain lies in $B^+:=B\cap\{x_d>0\}$. Let $B^-:=B\cap\{x_d<0\}$. Define the zero extension
\[
\widetilde u(x)=
\begin{cases}
u(x),&x\in B^+,\\
0,&x\in B^-.
\end{cases}
\]
We claim that
\[
-\Delta\widetilde u-\kappa^2\widetilde u=0
\qquad\text{in }\mathcal D'(B).
\]
Let $\phi\in C_c^\infty(B)$. Applying Green's formula on $B^+$ gives
\[
\langle-\Delta\widetilde u-\kappa^2\widetilde u,\phi\rangle_B
=
\langle\partial_{\bn}u,\gamma\phi\rangle_F
-\langle\gamma u,\partial_{\bn}\phi\rangle_F,
\]
where the normal is the exterior normal of $B^+$ on the flat interface $F$. The first term is zero by the assumed vanishing weak normal trace, and the second term is zero by the assumed vanishing Dirichlet trace. Therefore the zero extension satisfies the Helmholtz equation in the whole ball $B$ in the sense of distributions.

By interior elliptic regularity for the constant-coefficient operator $-\Delta-\kappa^2$, the distributional solution $\widetilde u$ is smooth, indeed real analytic, in $B$. Since $\widetilde u$ is identically zero in the nonempty open set $B^-$, analyticity implies $\widetilde u=0$ in $B$. Hence $u$ vanishes in the nonempty open set $B^+\subset\Omega$. Applying the same unique-continuation principle inside the connected domain $\Omega$ gives $u=0$ throughout $\Omega$.
\end{proof}

\section{Explicit Cui--Zhang-type error bounds}\label{app:explicit-CZ-error}

This appendix records the algebra that turns the projection-defect estimate into explicit error bounds of the same form as \cite[Corollary~3.2]{CuiZhang2014}. The local projection approximation estimates are stated as assumptions in the form needed for the calculation below.

Let
\[
\eta_q:=\|\bq-\pi\bq\|_{\T_h},
\qquad
\eta_u:=\|u-\Pi u\|_{\T_h}.
\]
Assume that the exact solution is sufficiently smooth and satisfies the wave-solution scaling assumptions used in the projection estimates. More precisely, assume that the following four bounds hold:
\begin{align}
\eta_q^2
&\le
C_\pi (h\kappa)^{2p+2}
(1+\tau_{\max}\kappa^{-1})^2
\|\bq\|_{\T_h}^2,
\label{eq:app-proj-q-by-q}\\
\eta_u^2
&\le
C_\pi (h\kappa)^{2p+2}
(\kappa^{-1}+\kappa\tau_{\min}^{-1})^2
\|\bq\|_{\T_h}^2,
\label{eq:app-proj-u-by-q}\\
\eta_q^2
&\le
C_\pi (h\kappa)^{2p+2}
\kappa^2(1+\tau_{\max}\kappa^{-1})^2
\|u\|_{\T_h}^2,
\label{eq:app-proj-q-by-u}\\
\eta_u^2
&\le
C_\pi (h\kappa)^{2p+2}
(1+\kappa^2\tau_{\min}^{-1})^2
\|u\|_{\T_h}^2.
\label{eq:app-proj-u-by-u}
\end{align}
These estimates are the projection approximation input used to pass from a projection-defect estimate to the explicit Cui--Zhang-type bounds.

\begin{proposition}\label{prop:CZ-type-error}
Under the assumptions of \cref{thm:main-stability} and the projection approximation estimates \eqref{eq:app-proj-q-by-q}--\eqref{eq:app-proj-u-by-u}, the HDG solution with $Q=0$ satisfies
\begin{align}
\|\bq-\bq_h\|_{\T_h}^2
&\le
C_1(h\kappa)^{2p+2}(1+\tau_{\max}\kappa^{-1})^2\|\bq\|_{\T_h}^2
\notag\\
&\quad+
C_2 C_{\rm st}(h\kappa)^{2p+2}
\left[(1+\tau_{\max}\kappa^{-1})^2+
\kappa^4(\kappa^{-1}+\kappa\tau_{\min}^{-1})^2\right]
\|\bq\|_{\T_h}^2,
\label{eq:app-CZ-type-q}
\end{align}
and
\begin{align}
\|u-u_h\|_{\T_h}^2
&\le
C_1(h\kappa)^{2p+2}(1+\kappa^2\tau_{\min}^{-1})^2\|u\|_{\T_h}^2
\notag\\
&\quad+
C_2 C_{\rm st}(h\kappa)^{2p+2}
\left[(1+\tau_{\max}\kappa^{-1})^2+
\kappa^4(\kappa^{-1}+\kappa\tau_{\min}^{-1})^2\right]
\|u\|_{\T_h}^2.
\label{eq:app-CZ-type-u}
\end{align}
The constants may depend on the fixed polynomial degree, the shape-regularity class, the prescribed compact wave-number interval, and the constants in \eqref{eq:app-proj-q-by-q}--\eqref{eq:app-proj-u-by-u}; the stability factor contains no negative powers of $h$.
\end{proposition}

\begin{proof}
We first estimate the flux error. By the triangle inequality,
\[
\|\bq-\bq_h\|_{\T_h}^2
\le
2\|\bq-\pi\bq\|_{\T_h}^2
+
2\|\pi\bq-\bq_h\|_{\T_h}^2
=
2\eta_q^2+
2\|\pi\bq-\bq_h\|_{\T_h}^2.
\]
The projection error estimate \eqref{eq:projection-error-estimate} gives
\[
\|\pi\bq-\bq_h\|_{\T_h}^2
\le
C_{\rm st}
\left(
\eta_q^2+\kappa^4\eta_u^2
\right).
\]
Therefore,
\[
\|\bq-\bq_h\|_{\T_h}^2
\le
2\eta_q^2+
2C_{\rm st}
\left(
\eta_q^2+\kappa^4\eta_u^2
\right).
\]
Substituting \eqref{eq:app-proj-q-by-q} and \eqref{eq:app-proj-u-by-q} yields \eqref{eq:app-CZ-type-q} after renaming constants.

For the scalar error, the triangle inequality gives
\[
\|u-u_h\|_{\T_h}^2
\le
2\|u-\Pi u\|_{\T_h}^2+
2\|\Pi u-u_h\|_{\T_h}^2
=
2\eta_u^2+2\|\Pi u-u_h\|_{\T_h}^2.
\]
Again by \eqref{eq:projection-error-estimate},
\[
\kappa^2\|\Pi u-u_h\|_{\T_h}^2
\le
C_{\rm st}
\left(
\eta_q^2+\kappa^4\eta_u^2
\right),
\]
and hence
\[
\|\Pi u-u_h\|_{\T_h}^2
\le
C_{\rm st}
\left(
\kappa^{-2}\eta_q^2+\kappa^2\eta_u^2
\right).
\]
Thus
\[
\|u-u_h\|_{\T_h}^2
\le
2\eta_u^2+
2C_{\rm st}
\left(
\kappa^{-2}\eta_q^2+\kappa^2\eta_u^2
\right).
\]
Substituting \eqref{eq:app-proj-q-by-u} and \eqref{eq:app-proj-u-by-u} gives
\[
\kappa^{-2}\eta_q^2
\le
C_\pi (h\kappa)^{2p+2}
(1+\tau_{\max}\kappa^{-1})^2
\|u\|_{\T_h}^2,
\]
and
\[
\kappa^2\eta_u^2
\le
C_\pi (h\kappa)^{2p+2}
\kappa^2(1+\kappa^2\tau_{\min}^{-1})^2
\|u\|_{\T_h}^2.
\]
Since
\[
\kappa^2(1+\kappa^2\tau_{\min}^{-1})^2
=
\kappa^4(\kappa^{-1}+\kappa\tau_{\min}^{-1})^2,
\]
we obtain \eqref{eq:app-CZ-type-u} after renaming constants.
\end{proof}

\section{Additional numerical results}\label{app:numerics}

This appendix records supplementary numerical data for the experiments in
\Cref{sec:numerics}. The results here are not needed for the proof, but they document
the reproducibility and the behavior of the method in the tests described in the
main text.

\subsection{Detailed wave-number sweep}\label{app:num-kappa}

\Cref{tab:app-kappa-fixed-ppw} gives the fixed-points-per-wavelength part of the
wave-number sweep. The mesh size is chosen so that the number of points per
wavelength is approximately $12.57$ for $\kappa\ge4$ and is at least $N=8$ for
low wave numbers. The table shows the expected growth of the Helmholtz error with
$\kappa$ and the favorable high-frequency behavior of the choice $\tau=\kappa$ in
the scalar variable. This experiment is included to clarify the interpretation of the
locally uniform-in-$\kappa$ theorem: the method is not being claimed to be
high-frequency uniform.

\begin{table}[tbp]
\centering
\caption{Wave-number sweep with fixed points per wavelength.}
\label{tab:app-kappa-fixed-ppw}
\begin{tabular}{cccccc}
\toprule
$\kappa$ & $\tau$ & $N$ & points per wavelength &
$\|u-u_h\|_{\T_h}$ & $\|\bq-\bq_h\|_{\T_h}$ \\
\midrule
1  & $1$          & 8  & 50.27 & $3.46\times10^{-4}$ & $4.78\times10^{-4}$ \\
1  & $\kappa$    & 8  & 50.27 & $3.46\times10^{-4}$ & $4.78\times10^{-4}$ \\
1  & $0.1\kappa$ & 8  & 50.27 & $2.10\times10^{-3}$ & $3.57\times10^{-4}$ \\
1  & $10\kappa$  & 8  & 50.27 & $2.82\times10^{-4}$ & $3.21\times10^{-3}$ \\
2  & $1$          & 8  & 25.13 & $2.01\times10^{-3}$ & $3.10\times10^{-3}$ \\
2  & $\kappa$    & 8  & 25.13 & $1.39\times10^{-3}$ & $3.83\times10^{-3}$ \\
2  & $0.1\kappa$ & 8  & 25.13 & $8.30\times10^{-3}$ & $2.91\times10^{-3}$ \\
2  & $10\kappa$  & 8  & 25.13 & $1.21\times10^{-3}$ & $2.47\times10^{-2}$ \\
4  & $1$          & 8  & 12.57 & $1.41\times10^{-2}$ & $2.40\times10^{-2}$ \\
4  & $\kappa$    & 8  & 12.57 & $5.62\times10^{-3}$ & $3.07\times10^{-2}$ \\
4  & $0.1\kappa$ & 8  & 12.57 & $3.19\times10^{-2}$ & $2.71\times10^{-2}$ \\
4  & $10\kappa$  & 8  & 12.57 & $6.70\times10^{-3}$ & $1.76\times10^{-1}$ \\
8  & $1$          & 16 & 12.57 & $2.66\times10^{-2}$ & $6.29\times10^{-2}$ \\
8  & $\kappa$    & 16 & 12.57 & $5.81\times10^{-3}$ & $6.24\times10^{-2}$ \\
8  & $0.1\kappa$ & 16 & 12.57 & $3.23\times10^{-2}$ & $6.95\times10^{-2}$ \\
8  & $10\kappa$  & 16 & 12.57 & $1.03\times10^{-2}$ & $3.57\times10^{-1}$ \\
16 & $1$          & 32 & 12.57 & $4.95\times10^{-2}$ & $3.13\times10^{-1}$ \\
16 & $\kappa$    & 32 & 12.57 & $6.57\times10^{-3}$ & $1.34\times10^{-1}$ \\
16 & $0.1\kappa$ & 32 & 12.57 & $3.40\times10^{-2}$ & $2.26\times10^{-1}$ \\
16 & $10\kappa$  & 32 & 12.57 & $1.92\times10^{-2}$ & $7.56\times10^{-1}$ \\
32 & $1$          & 64 & 12.57 & $8.71\times10^{-2}$ & $1.78$ \\
32 & $\kappa$    & 64 & 12.57 & $9.02\times10^{-3}$ & $3.33\times10^{-1}$ \\
32 & $0.1\kappa$ & 64 & 12.57 & $4.04\times10^{-2}$ & $8.44\times10^{-1}$ \\
32 & $10\kappa$  & 64 & 12.57 & $3.77\times10^{-2}$ & $1.82$ \\
\bottomrule
\end{tabular}
\end{table}

The random-right-hand-side stability ratios in the same wave-number sweep remain
of order one. On the fixed mesh $N=128$, the largest $R_{\max}$ over all tested
$\kappa$ and $\tau$ is $0.640$. Under the fixed-points-per-wavelength rule, the
largest value is $0.786$, occurring at $\kappa=32$ and $\tau=10\kappa$.

\subsection{Rough \texorpdfstring{$L^2$}{L2} forcing}\label{app:num-rough}

The rough-data experiment uses
\[
f=\chi_{D_1}+\ii\chi_{D_2},
\qquad
g=0,
\]
with
\[
D_1=[1/4,3/4]\times[1/4,3/4],
\qquad
D_2=[1/8,5/8]\times[3/8,7/8].
\]
A reference solution is computed with $N_{\rm ref}=256$ and $p_{\rm ref}=2$.
\Cref{tab:app-rough} reports the reference errors. The errors decrease monotonically,
and the energy norm \eqref{eq:num-energy-ratio} remains bounded throughout the test.

\begin{table}[tbp]
\centering
\caption{Rough $L^2$ forcing: errors with respect to the reference solution
computed with $N_{\rm ref}=256$ and $p_{\rm ref}=2$.}
\label{tab:app-rough}
\begin{tabular}{ccccccc}
\toprule
$p$ & $N$ &
$\|u_h-u_{\rm ref}\|_{\T_h}$ & rate &
$\|\bq_h-\bq_{\rm ref}\|_{\T_h}$ & rate &
$E_h$ \\
\midrule
0 & 8   & $2.49\times10^{-2}$ & --    & $9.99\times10^{-2}$ & --    & 0.251 \\
0 & 16  & $1.51\times10^{-2}$ & 0.718 & $6.01\times10^{-2}$ & 0.733 & 0.285 \\
0 & 32  & $8.32\times10^{-3}$ & 0.865 & $3.28\times10^{-2}$ & 0.874 & 0.306 \\
0 & 64  & $4.34\times10^{-3}$ & 0.937 & $1.71\times10^{-2}$ & 0.942 & 0.318 \\
0 & 128 & $2.22\times10^{-3}$ & 0.970 & $8.70\times10^{-3}$ & 0.973 & 0.325 \\
1 & 8   & $1.55\times10^{-3}$ & --    & $4.36\times10^{-3}$ & --    & 0.329 \\
1 & 16  & $3.85\times10^{-4}$ & 2.010 & $1.04\times10^{-3}$ & 2.062 & 0.331 \\
1 & 32  & $9.58\times10^{-5}$ & 2.006 & $2.65\times10^{-4}$ & 1.977 & 0.331 \\
1 & 64  & $2.39\times10^{-5}$ & 2.002 & $6.86\times10^{-5}$ & 1.952 & 0.331 \\
1 & 128 & $5.97\times10^{-6}$ & 2.001 & $1.78\times10^{-5}$ & 1.949 & 0.331 \\
\bottomrule
\end{tabular}
\end{table}

\subsection{Three-dimensional verification}\label{app:num-3d}

The three-dimensional test uses
\[u(x,y,z)=
\exp\!\left(\ii\kappa\frac{x+y+z}{\sqrt3}\right),
\qquad
\kappa=3,
\]
on $\Omega=(0,1)^3$. The mesh is the uniform tetrahedral mesh of the cube, and
we test $p=0,1$ with $\tau=1$ and $\tau=\kappa$. The results in
\Cref{tab:app-3d} verify the expected rates in three dimensions. The largest
algebraic residual in this experiment is $2.43\times10^{-13}$, the largest weak
internal flux residual is $7.08\times10^{-12}$, and the largest weak Robin residual is
$1.84\times10^{-12}$.

\begin{table}[tbp]
\centering
\caption{Three-dimensional plane-wave verification, $\kappa=3$.}
\label{tab:app-3d}
\small
\begin{tabular}{ccccccc}
\toprule
$p$ & $\tau$ & $N$ &
$\|u-u_h\|_{\T_h}$ & rate &
$\|\bq-\bq_h\|_{\T_h}$ & rate \\
\midrule
0 & $1$       & 4  & $2.97\times10^{-1}$ & --    & $8.72\times10^{-1}$ & -- \\
0 & $1$       & 8  & $1.57\times10^{-1}$ & 0.917 & $4.62\times10^{-1}$ & 0.917 \\
0 & $1$       & 16 & $8.04\times10^{-2}$ & 0.965 & $2.36\times10^{-1}$ & 0.966 \\
0 & $\kappa$ & 4  & $2.94\times10^{-1}$ & --    & $1.15$ & -- \\
0 & $\kappa$ & 8  & $1.60\times10^{-1}$ & 0.874 & $6.40\times10^{-1}$ & 0.840 \\
0 & $\kappa$ & 16 & $8.33\times10^{-2}$ & 0.945 & $3.37\times10^{-1}$ & 0.926 \\
1 & $1$       & 4  & $3.35\times10^{-2}$ & --    & $9.80\times10^{-2}$ & -- \\
1 & $1$       & 8  & $8.32\times10^{-3}$ & 2.008 & $2.42\times10^{-2}$ & 2.021 \\
1 & $1$       & 16 & $2.07\times10^{-3}$ & 2.004 & $5.99\times10^{-3}$ & 2.011 \\
1 & $\kappa$ & 4  & $2.84\times10^{-2}$ & --    & $1.48\times10^{-1}$ & -- \\
1 & $\kappa$ & 8  & $6.89\times10^{-3}$ & 2.044 & $3.59\times10^{-2}$ & 2.040 \\
1 & $\kappa$ & 16 & $1.70\times10^{-3}$ & 2.016 & $8.79\times10^{-3}$ & 2.031 \\
\bottomrule
\end{tabular}
\end{table}

\subsection{Acceptance diagnostics}\label{app:num-diagnostics}

For completeness, \Cref{tab:app-diagnostics} records the largest residual
diagnostics observed in the full set of numerical experiments. All values are weak
residuals in the finite element test spaces, except for the algebraic residual of the
condensed linear system.

\begin{table}[h!]
\centering
\caption{Maximum residual diagnostics over the numerical experiments.}
\label{tab:app-diagnostics}
\begin{tabular}{lccc}
\toprule
experiment &
algebraic residual &
$C_{\rm int}$ &
$C_{\rm bdry}$ \\
\midrule
plane-wave convergence       & $1.58\times10^{-11}$ & $1.00\times10^{-10}$ & $1.67\times10^{-11}$ \\
mesh-uniform stability       & $9.54\times10^{-12}$ & $4.51\times10^{-13}$ & $7.25\times10^{-14}$ \\
$Q$-forcing stability        & $1.75\times10^{-11}$ & $5.55\times10^{-13}$ & $8.12\times10^{-14}$ \\
wave-number sweep            & $4.41\times10^{-12}$ & $1.77\times10^{-11}$ & $3.13\times10^{-12}$ \\
rough $L^2$ forcing          & $9.09\times10^{-12}$ & $5.58\times10^{-13}$ & $4.43\times10^{-14}$ \\
three-dimensional check      & $2.43\times10^{-13}$ & $7.08\times10^{-12}$ & $1.84\times10^{-12}$ \\
\bottomrule
\end{tabular}
\end{table}

The small residuals reported in \Cref{tab:app-diagnostics} confirm that the
convergence and stability results are not contaminated by poorly solved linear
systems or by violations of the weak HDG conservation and boundary equations.

\bibliographystyle{abbrv}
\bibliography{reference}

\end{document}